\documentclass[12pt, english]{article}
\usepackage[utf8]{inputenc}
\usepackage{amsthm}
\usepackage{amssymb}
\usepackage[all]{xy}
\usepackage{amsmath}
\usepackage{tikz-cd}
\usepackage{tikz}
\usepackage{float}

\tikzcdset{scale cd/.style={every label/.append style={scale=#1},
		cells={nodes={scale=#1}}}}
\usepackage{rotating}
\usepackage{caption, booktabs}
\usepackage{verbatim}
\usetikzlibrary{calc}

\usepackage{tikz}
\usetikzlibrary{patterns}

\usepackage{pdflscape}

\usepackage{mathtools}
\usepackage[english]{babel}

\usepackage{hyperref}
\usepackage{bookmark}

\newtheorem{definition}{Definition}[section]
\newtheorem{proposition}[definition]{Proposition}
\newtheorem{remark}[definition]{Remark}
\newtheorem{theorem}[definition]{Theorem}
\newtheorem{construction}[definition]{Construction}
\newtheorem{example}[definition]{Example}

\DeclareMathOperator{\ord}{ord}
\DeclareMathOperator{\Hom}{Hom}

\DeclareMathOperator{\rk}{rk}
\DeclareMathOperator{\Cl}{Cl}

\begin{document}
	\title{Kanev and Todorov surfaces in toric 3-folds}
	\author{Julius Giesler \\ University of T\"ubingen}
	\date{\today}
	\maketitle
	\begin{abstract} 
	In the first part of this article we show for some examples of surfaces of general type in toric $3$-folds how to construct minimal and canonical models by toric methods explicitly. The examples we study turn out to be surfaces of general type, namely so called Kanev surfaces and Todorov surfaces. We show how properties of our examples of surfaces could be derived directly from properties of some polytopes and we compute the singularities of their canonical models.
	\end{abstract}
	
	\maketitle
	\section{Introduction}
	In this article we study some surfaces of general type that arise as hypersurfaces in toric 3-folds. This article illustrates results from the articles (\cite{Bat20}) and (\cite{Gie21}) via concrete examples. \\
	We start with some lattice polytope $\Delta$ and a Laurent polynomial $f$ which has $\Delta$ as its Newton polytope and is nondegenerate with respect to $\Delta$. For this we take $\Delta$ to be a $3$-dimensional canonical Fano polytope, that is $\Delta$ contains just $(0,0,0)$ as an interior lattice point and the vertices of $\Delta$ are primitive lattice vectors. Additionaly since there are still quite a lot of such polytopes we restrict to those with 
	\[ \dim \, F(\Delta) = 3, \]
	where $F(\Delta)$ denotes the Fine interior of $\Delta$ (see Definition \ref{definition_Fine_interior}). $F(\Delta)$ is a rational polytope that can be constructed from the Newton polytope $\Delta$. The restriction to $\dim \, F(\Delta) = 3$ has the effect that if we let $Z_f := \{f=0\} \subset (\mathbb{C}^*)^3$, then $Z_f$ is birational to a surface of general type. \\
	There are $49$ canonical Fano polytopes $\Delta$ with $\dim \, F(\Delta) = 3$. Then we ask how to find a minimal model of $Z_f$, that is a smooth projective surface $Y$, which is birational to $Z_f$ and with $K_Y$ nef. \\
	For this we recall in section \ref{section_toric_methods_Kanev_surfaces} the Definition of the Fine interior $F(\Delta)$, the canonical closure $C(\Delta)$, the Minkowski sum 
	\[ \tilde{\Delta} := C(\Delta) + F(\Delta)  \]
	and of a simplicial fan $\Sigma$ refining the normal fan of $\Sigma_{\tilde{\Delta}}$. 
	These constructions are necessary in order to specify the toric variety in which a minimal model of the surface $Z_f$ is contained.	To be more precise we take the projective toric varieties to the normal fans of $F(\Delta)$, $C(\Delta)$ and $\tilde{\Delta}$ as well as the projective toric variety to the fan $\Sigma$ and call them $\mathbb{P}_{F(\Delta)}, \, \mathbb{P}_{C(\Delta)}, \, \mathbb{P}_{\tilde{\Delta}}$ and $\mathbb{P}_{\Sigma}$. Then we obtain a diagram
		\begin{equation}  \label{diagram_toric_morphisms}
		\begin{tikzcd}
			& \mathbb{P}_{\Sigma} \arrow[swap]{d}{\pi} \\
			& \mathbb{P}_{\tilde{\Delta}} \arrow[swap]{dl}{\rho} \arrow{dr}{\theta} \\
			\mathbb{P}_{C(\Delta)} && \mathbb{P}_{F(\Delta)}
		\end{tikzcd}
	\end{equation}
	
	of birational toric morphisms, where $\theta$ is birational since we assume \\ $\dim \, F(\Delta) = 3$. Let $Z_{\Sigma}$ denote the closure of $Z_f$ in $\mathbb{P}_{\Sigma}$, then $Z_{\Sigma}$ gets a minimal surface of general type and of the closure $Z_{F(\Delta)}$ of $Z_f$ in $\mathbb{P}_{F(\Delta)}$ gets a canonical model of $Z_{\Sigma}$. Since we always assume $f$ to be nondegenerate with respect to its Newton polytope $\Delta$ we are able to compute the singularities of $Z_{F(\Delta)}$ and of the closure $Z_{\tilde{\Delta}} \subset \mathbb{P}_{\tilde{\Delta}}$ of $Z_f$ at least if $\Delta$ is canonically closed, that is $\Delta = C(\Delta)$. \\
	Apart from just computing the singularities we also illustrate these results as well as the $49$ polytopes in some pictures. Besides	we provide some combinatorial classification of these $49$ polytopes, that is we bring them onto a normal form and divide them into $5$ classes $a), \, b), \, c), \, d)$ and $e)$ which could be studied separately.
	\\
	In section \ref{chapter_constructing_Kanev_Todorov_surfaces} we study invariants of the resulting minimal surfaces and in this way identify our surfaces with so called \textit{Kanev } surfaces in the cases $a)$ and $b)$ and \textit{Todorov} surfaces in the cases $c)$, $d)$ and $e)$.
	\\
	Kanev surfaces (also known as Kunev or Kynev surfaces) have invariants
	\[ p_{g}(Y) = 1, \quad K_Y^2 = 1, \]
	and were studied in (\cite{Cat78}, \cite{Us87}, \cite{Tod80}). Todorov surfaces have invariants
	\[ p_{g}(Y) = 1, \quad q(Y) = 0, \quad K_Y^2 = 2 \]
	and were studied in (\cite{CD89}). Both of these surfaces became interesting from the point of view of Torelli type Theorems: Some of them for example fail the infinitesimal Torelli Theorem.  \\
	We remark that the plurigenera of $Z_{\Sigma}$ could be calculated via $2$ different methods: Either via general formulas for plurigenera of minimal algebraic surfaces of general type or by counting lattice points and interior lattice points on $F(\Delta)$. It is known that for $Y$ a Kanev or Todorov surface $2K_Y$ is basepointfree (\cite{Cat78}, \cite{CD89}). If $Y=Z_{\Sigma}$ lies in the toric $3$-fold $\mathbb{P}_{\Sigma}$ this is gotten for free by the adjunction formula and an argument using a resolution of singularities of $\mathbb{P}_{\Sigma}$ (see \cite[Prop.6.2]{Gie21}). We further deal with Kanev's original example of a Kanev surface.

	\section{Combinatorial classification of 49 canonical Fano polytopes \texorpdfstring{$\Delta$}{x}  with \texorpdfstring{\bfseries $\dim \, F(\Delta) = 3$}{x}}
	\label{section_toric_methods_Kanev_surfaces}
	\subsection{General background and the Fine interior} \label{subsection_general_background_Fine_interior}
	\underline{Notation} (Toric setting):
	\begin{itemize}
			\setlength\itemsep{0em}
		\item $M \cong \mathbb{Z}^3$: A $3$-dimensional lattice with dual lattice $N$, $T := \Hom(M, \mathbb{C}^*) \cong (\mathbb{C}^*)^3$ the torus. We identify $M$ with the lattice of characters of $T$.
		\item $\Delta \subset M_{\mathbb{R}}:= M \otimes \mathbb{R}$: A $3$-dimensional lattice polytope.
		\item $\Sigma_{\Delta}$: The normal fan to the polytope $\Delta$ and more generally $\Sigma_F$ denotes the normal fan to a rational polytope $F \subset M_{\mathbb{R}}$.
		\item For $\Sigma$ a fan in $N_{\mathbb{R}}$ let $\Sigma[i]$ denote the set of $i$-dimensional cones of $\Sigma$ or for $i=1$ also the set of generators of the $1$-dimensional cones.
		\item $\mathbb{P}_{\Delta}$: The toric variety to the polytope $\Delta$, defined via its normal fan, more generally $\mathbb{P}_{\Sigma}$: The toric variety to a fan $\Sigma$.
		\item $\langle v_0,...,v_n \rangle$: The convex span of the lattice points $v_0,...,v_n$.
	\end{itemize}

	Recall that to a ray $\nu_i \in \Sigma_{\Delta}[1]$ is associated a torus invariant divisor $D_i$ and that the canonical divisor of $\mathbb{P}_{\Delta}$ is given by 
	\begin{align} \label{formula_canonical_divisor_toric_variety}
	 K_{\mathbb{P}_{\Delta}} = - \sum\limits_{\nu_i \in \Sigma_{\Delta}[1]} D_i. \end{align}
	In this article we just deal with integral divisors. We are mainly interested in hypersurfaces in toric varieties and for this we take $\Delta$ to be the Newton polytope of a Laurent polynomial $f$:
	\begin{definition} \label{definition_Newton_polytope}
		Let $f$ be a Laurent polynomial with presentation 
		\begin{align} \label{f_Laurent_polynomial}
		f = \sum\limits_{m \in A} a_m z^{m}, \quad a_i \in \mathbb{C}
		\end{align}
		for some finite set $A \subset M$. The convex span of the $m \in A$ with $a_m \neq 0$ is called the Newton polytope of $f$.
	\end{definition}
	$L(\Delta)$ denotes the set of Laurent polynomials $f$ as in (\ref{f_Laurent_polynomial}) with Newton polytope $\Delta$ and
	\[ l(\Delta) := \textrm{ dim}_{\mathbb{C}} \, L(\Delta) = |M \cap \Delta| \]
	the number of lattice points in $\Delta$. Let $l^{*}(\Delta)$ be the number of interior lattice points of $\Delta$ and $Z_{f}$ the zero set $\{f=0\} \subset T$.
	\begin{definition}
		In the above situation $f$ is called $\Delta$-regular or nondegenerate with respect to $\Delta$ if $Z_f$ is smooth and for every face $\Gamma \subset \Delta$ the variety $Z_{f \vert{\Gamma}}$ is smooth as well, where
		\[ f_{\vert{\Gamma}} := \sum\limits_{m \in \Gamma \cap M} a_m z^{m}. \]
	\end{definition}
For $f \in L(\Delta)$ we denote the closure of $Z_f$ in the toric variety $\mathbb{P}_{\Delta}$ by $Z_{\Delta}$. The nondegeneracy means that $Z_{\Delta}$ intersects the toric strata of $\mathbb{P}_{\Delta}$ transversally in a subset of codimension $1$. For $f$ nondegenerate we sometimes also call $Z_{\Delta}$ nondegenerate. Given a complete $3$-dimensional fan $\Sigma$ we write $Z_{\Sigma}$ for the closure of $Z_f$ in $\mathbb{P}_{\Sigma}$ and if $\Sigma = \Sigma_F$ is the normal fan of some rational polytope $F$ we write $Z_F$ for $Z_{\Sigma_F}$. We call $Z_{\Sigma}$ nondegenerate (with respect to $\Sigma$) if it intersects the toric strata of $\mathbb{P}_{\Sigma}$ transversally.

\begin{construction} \label{construction_divisors_polytopes}
		\normalfont
	For $\nu \in N$ let
	\[ \ord_{\Delta}(\nu):= \min\limits_{m \in \Delta \cap M} \langle m, \nu \rangle, \]
	then $\Delta$ has a facet presentation
	\begin{align} \label{facet_presentation_polytope_Delta}
	 \Delta = \{x \in M_{\mathbb{R}}| \, \langle x,\nu_i \rangle \geq \ord_{\Delta}(\nu_i) \}. 
	 \end{align}
	with $ \nu_{1},...,\nu_{s} \in \Sigma_{\Delta}[1]$ the inner facet normals. By (\cite[Prop.5.1]{Bat20}) $Z_{\Delta}$ is linear equivalent to the following torus invariant divisor
	\begin{align} \label{linear_equivalence_hypersurface_ample_sheaf}
	Z_{\Delta} \sim_{lin}  -\sum\limits_{\nu_i \in \Sigma_{\Delta}[1]} \ord_{\Delta}(\nu_i) D_i.
	\end{align}

\end{construction}

	\begin{definition}
	A lattice polytope $\Delta \subset M_{\mathbb{R}}$ is said to be
	\begin{itemize}
		\item canonical, if $l^{*}(\Delta) = 1$,
		\item Fano, if the vertices of $\Delta$ are primitive lattice vectors.
	\end{itemize}
	\end{definition}

	By the classification in (\cite{Kas10}) there are still $674.688$ canonical Fano $3$-topes, but these polytopes could be further divided into classes in terms of their Fine interior:
	\begin{definition} \label{definition_Fine_interior} (\cite[Appendix to §4]{Rei87})\\
		The Fine interior of $\Delta$ is defined as
		\[ F(\Delta) := \{ x \in M_{\mathbb{R}}| \, \langle x, \nu \rangle \geq \ord_{\Delta}(\nu) + 1, \, \nu \in N \setminus \{0\} \} \]
	\end{definition}
	
	\begin{remark}
		\normalfont
		In order to construct the Fine interior $F(\Delta)$ of $\Delta$ we have to move every hyperplane which touches some face of $\Delta$ one into the interior of $\Delta$.	In general it is not enough to move just the hyperplanes defining facets one into the interior (see Figure \ref{figure_examples_polytopes_Fine_interior}). However finitely many hyperplanes will always be enough. In fact define the support $S_{F}(\Delta)$ of $F(\Delta)$ to $\Delta$ as follows:
		\end{remark}
		\begin{definition}
			\normalfont
			The set of lattice points $\nu \in N \setminus \{0\}$ with
			\[ \ord_{F(\Delta)}(\nu) = \ord_{\Delta}(\nu) + 1 \]
			is called the support of $F(\Delta)$ to $\Delta$ and is denoted by $S_{F}(\Delta)$. Then we have
		\end{definition}

		\begin{proposition} \label{proposition_support_finite_number_of_points} (\cite[Prop.1.11]{Bat20})
			\[ S_{F}(\Delta) \subset \langle \nu_i | \, \nu_i \in \Sigma_{\Delta}[1] \rangle \]
			In particular there are only finitely many hyperplanes touching $\Delta$ and after moving into the interior also touching $F(\Delta)$.
		\end{proposition}
	\begin{remark} \label{remark_Fine_interior_2_dim_at_least_3_dim}
		\normalfont
				For $n=2$ the Fine interior of $\Delta$ always equals the convex span of the interior lattice points of $\Delta$, but for $n \geq 3$ the Fine interior $F(\Delta)$ is in general only a rational polytope, i.e. the vertices have rational coordinates (\cite[Prop.2.9, Rem. 2.10]{Bat17}).
	\end{remark}

\begin{definition} \cite[Def.1.13]{Bat20} \\
	\normalfont
	The polytope
	\[ C(\Delta) := \{ x \in M_{\mathbb{R}} \, | \, \langle x,\nu \rangle \geq \ord_{\Delta}(\nu) \quad  \forall \, \nu \in S_{F}(\Delta) \} \]
	is called the canonical closure of $\Delta$. We call $\Delta$ canonically closed if $C(\Delta) = \Delta$. 
\end{definition}

\begin{remark} \normalfont
	Certainly the canonical closure $C(\Delta)$ contains $\Delta$ and for dim$(\Delta) = 2$ with $F(\Delta) \neq \emptyset$ we even have equality $C(\Delta) = \Delta$ (see \cite[Prop.2.4]{Bat20}). In general we have (\cite[Prop.1.17(b), Cor.1.19]{Bat20}):
	\begin{align}\label{Fine_interior_canonical_closure}
	F(C(\Delta)) = F(\Delta), \quad C(C(\Delta)) = C(\Delta),
	\end{align}
	that is $C(\Delta)$ is canonically closed and still has the same Fine interior as $\Delta$.

\end{remark}
	
	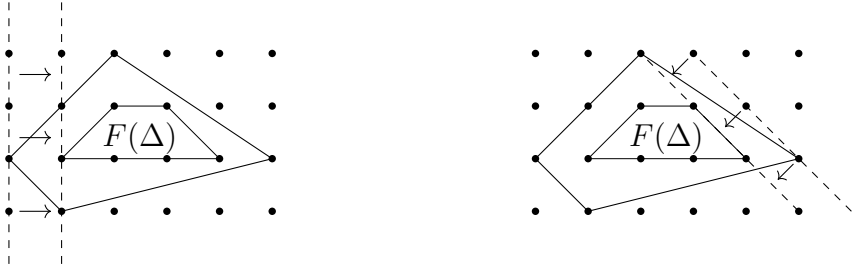
\begin{figure}
	\begin{tikzpicture}[scale=0.7]
	
	\fill (-1,1) circle(2pt);
	\fill (-1,2) circle(2pt);
	\fill (-1,3) circle(2pt);
	\fill (-1,4) circle(2pt);
	\fill (0,1) circle (2pt);
	\fill (0,2) circle (2pt);	
	\fill (0,3) circle (2pt);
	\fill (0,4) circle (2pt);	
	\fill (1,1) circle (2pt);
	\fill (1,2) circle (2pt);	
	\fill (1,3) circle (2pt);
	\fill (1,4) circle (2pt);	
	\fill (2,1) circle (2pt);
	\fill (2,2) circle (2pt);	
	\fill (2,3) circle (2pt);
	\fill (2,4) circle (2pt);	
	\fill (3,1) circle (2pt);
	\fill (3,2) circle (2pt);	
	\fill (3,3) circle (2pt);
	\fill (3,4) circle (2pt);	
	\fill (4,1) circle (2pt);
	\fill (4,2) circle (2pt);	
	\fill (4,3) circle (2pt);
	\fill (4,4) circle (2pt);
	
	\draw (-1,2) -- (1,4);
	\draw (1,4) -- (4,2);
	\draw (-1,2)  -- (0,1);
	\draw (0,1) -- (4,2);
	
	\draw[very thin] (0,2)--(1,3);
	\draw[very thin] (1,3) -- (2,3);
	\draw[very thin] (2,3) -- (3,2);
	\draw[very thin] (3,2) -- (0,2);
	
	\draw[dashed] (-1,0) -- (-1,5);	
	\draw[arrows=->](-0.8,1)--(-0.2,1);
	\draw[arrows=->](-0.8,2.4)--(-0.2,2.4);
	\draw[arrows=->](-0.8,3.6)--(-0.2,3.6);
	\draw[dashed] (0,0) -- (0,5);

		\node [left] at (2.4,2.4) {{\bf $F(\Delta)$ }};
	
	\begin{scope}[xshift = 10cm]
	
		\fill (-1,1) circle(2pt);
	\fill (-1,2) circle(2pt);
	\fill (-1,3) circle(2pt);
	\fill (-1,4) circle(2pt);
	\fill (0,1) circle (2pt);
	\fill (0,2) circle (2pt);	
	\fill (0,3) circle (2pt);
	\fill (0,4) circle (2pt);	
	\fill (1,1) circle (2pt);
	\fill (1,2) circle (2pt);	
	\fill (1,3) circle (2pt);
	\fill (1,4) circle (2pt);	
	\fill (2,1) circle (2pt);
	\fill (2,2) circle (2pt);	
	\fill (2,3) circle (2pt);
	\fill (2,4) circle (2pt);	
	\fill (3,1) circle (2pt);
	\fill (3,2) circle (2pt);	
	\fill (3,3) circle (2pt);
	\fill (3,4) circle (2pt);	
	\fill (4,1) circle (2pt);
	\fill (4,2) circle (2pt);	
	\fill (4,3) circle (2pt);
	\fill (4,4) circle (2pt);
	
	\draw (-1,2) -- (1,4);
	\draw (1,4) -- (4,2);
	\draw (-1,2)  -- (0,1);
	\draw (0,1) -- (4,2);
	
	\draw[very thin] (0,2)--(1,3);
	\draw[very thin] (1,3) -- (2,3);
	\draw[very thin] (2,3) -- (3,2);
	\draw[very thin] (3,2) -- (0,2);
	
	\draw[dashed] (5,1) -- (2,4);	
	\draw[arrows=->](3.9,1.9)--(3.6,1.6);
	\draw[arrows=->](2.9,2.9)--(2.6,2.6);
	\draw[arrows=->](1.9,3.9)--(1.6,3.6);
	\draw[dashed] (4,1) -- (1,4);	
	
	\node [left] at (2.4,2.4) {{\bf $F(\Delta)$ }};

	\end{scope}
	\end{tikzpicture}
	\caption{Illustration of the construction of the Fine interior $F(\Delta)$ from $\Delta$.} \label{figure_examples_polytopes_Fine_interior}
\end{figure}

	\begin{construction} \label{specializing_facts_to_canonical_Fano_3_topes}
		\normalfont
	We restrict in this article to canonical Fano 3-topes $\Delta$ with 
	\[ \textrm{dim } F(\Delta) = 3. \]
	There are just $49$ such polytopes left, listed in (\cite[Appendix A.3]{Sch18}) and they share the following properties:
	\begin{itemize}
		\setlength\itemsep{0em}
		\item If $\Delta$ is canonically closed, then every facet of $\Delta$ has distance $1$ to the origin (=unique interior lattice point) except from one facet, which we call $\Delta_{can}$ with inner facet normal $\nu_{can}$, and which has distance $2$ to the origin, that is
		\[ \Delta = \{ x \in M_{\mathbb{R}}| \, \langle x, \nu_i \rangle \geq -1 \quad \nu_i \in \Sigma_{\Delta}[1] \setminus \nu_{can}, \, \langle x, \nu_{can} \rangle \geq -2 \} \]
		\item There are $46$ polytopes $\Delta$ with $l^{*}(\Delta_{can}) = 2$ and $3$ polytopes with $l^{*}(\Delta_{can}) = 3$.
	\end{itemize}		
	Thus $F(\Delta)$ has the origin as a vertex and exactly one facet, which we call $F(\Delta)_{can}$ opposite to $(0,0,0)$. For $\Delta \neq C(\Delta)$ there might exist more than one facet having distance $2$ to the origin (see for example the facet $\langle a_2, b,d \rangle$ in Table 1 ID:547524). \\
	It follows from the first point that if $\Delta = C(\Delta)$ then 
	\[ K_{Z_{\Delta}} = (Z_{\Delta} + K_{\mathbb{P}_{\Delta}})_{\vert{Z_{\Delta}}} = D_{can \vert{Z_{\Delta}}} \]
	could be identified with a curve on $D_{can}$ with Newton polytope $\Delta_{can}$. This curve is smooth for $f$ sufficiently nondegenerate.

\end{construction}

\subsection{Further classification of the 49 polytopes} \label{subsection_Further_classification_of_49_polytopes}

We brought the $49$ polytopes onto a normal form and showed by observation the following:
\begin{proposition}
	Among the $49$ polytopes there are just $5$ isomorphy types for the Fine interior. For all $49$ examples 
	\[ 2 \cdot F(\Delta)_{can} := \{ 2 \cdot t| t \in F(\Delta)_{can} \} \]
	is inscribed into $\Delta_{can}$ with the same shape as $\Delta_{can}$ (see the pictures below). Thus the $5$ different types of $F(\Delta)$ correspond to $5$ different types of $\Delta_{can}$.
\end{proposition}

For the first $46$ polytopes $\Delta$ with $l^{*}(\Delta_{can}) = 2$ there are two types of $\Delta_{can}$: For $20$ polytopes the facet $\Delta_{can}$ looks like in picture a) and for $26$ polytopes the facet looks like in picture b). We list these polytopes in the Tables \ref{table_spanning_set_polytope_Dcan_3vertices} and \ref{table_spanning_set_polytope_Dcan_4vertices}.

		\begin{center}
	\begin{tikzpicture}[scale = 0.9]
	
	\begin{scope}
	
	\draw[thick] (0,0) -- (-2,0);
	\draw[thick] (-2,0) -- (-1,-3);
	\draw[thick] (-1,-3) -- (0,0);
	
	\draw[thin] (-1,-2) -- (-2/3,-1);
	\draw[thin] (-2/3,-1) -- (-4/3,-1);
	\draw[thin] (-4/3,-1) -- (-1,-2);
	
	\fill (-2,-3) circle (3pt);
	\fill (-2,-2) circle (3pt);
	\fill (-2,-1) circle (3pt);
	\fill (-2,0) circle (3pt);
	\fill (-1,-3) circle (3pt);
	\fill (-1,-2) circle (3pt);	
	\fill (-1,-1) circle (3pt);
	\fill (-1,0) circle (3pt);
	\fill (0,-3) circle (3pt);
	\fill (0,-2) circle (3pt);	
	\fill (0,-1) circle (3pt);
	\fill (0,0) circle (3pt);
	
	
	\node [left] at (1.5,-4) {{\bf a) $\Delta_{can}$ and $2 \cdot F(\Delta)_{can}$}};

	\begin{scope}[xshift = 6cm, yshift = -1cm, rotate=90]

	\draw[thick] (-2,0) -- (0,1);
	\draw[thick] (0,1) -- (1,0);
	\draw[thick] (1,0) -- (0,-1);
	\draw[thick] (0,-1) -- (-2,0);
	
	\draw[thin] (-1,0) -- (-1/3,-1/3);
	\draw[thin] (-1/3,-1/3) -- (0,0);
	\draw[thin] (0,0) -- (-1/3,1/3);
	\draw[thin] (-1/3,1/3) -- (-1,0);		
	
	\fill (-2,-1) circle (3pt);
	\fill (-2,0) circle(3pt);
	\fill (-2,1) circle(3pt);
	\fill (-1,-1) circle (3pt);
	\fill (-1,0) circle (3pt);
	\fill (-1,1) circle (3pt);
	\fill (0,-1) circle (3pt);
	\fill (0,0) circle (3pt);	
	\fill (0,1) circle (3pt);
	\fill (1,-1) circle (3pt);
	\fill (1,0) circle (3pt);
	\fill (1,1) circle (3pt);	
	

	
	rotate = 0;
	\node [left] at (-3,-3) {{\bf b) $\Delta_{can}$ and $2 \cdot F(\Delta)_{can}$}};

	\end{scope}

	\end{scope}
	\end{tikzpicture}
\end{center}

For the remaining $3$ polytopes with $l^{*}(\Delta) = 3$ the facet $\Delta_{can}$ looks like in the pictures c), d) and e).
	\begin{center}
	\begin{tikzpicture}[scale = 0.9]

	\begin{scope}[yshift = 0cm, xshift = 0cm,rotate=180]
	
	\draw[thin] (1,4) -- (1/2,2);
	\draw[thin] (1/2,2) -- (3/2,2);
	\draw[thin] (3/2,2) -- (1,4);
	
	\draw[thick]  (1,5) --  (0,1);
	\draw[thick] (0,1) -- (2,1);
	\draw[thick] (2,1) -- (1,5);
	
	\fill (1,2) circle (3pt) node[right] {$y_1$};
	\fill (1,3) circle (3pt) node[right] {$y_2$};
	\fill (1,4) circle (3pt) node[right] {$y_3$};
	
	\fill (0,1) circle (3pt);
	\fill (0,2) circle (3pt);
	\fill (0,3) circle (3pt);
	\fill (0,4) circle (3pt);
	\fill (0,5) circle (3pt);
	\fill (1,1) circle (3pt);
	\fill (1,2) circle (3pt);
	\fill (1,3) circle (3pt);
	\fill (1,4) circle (3pt);
	\fill (1,5) circle (3pt);
	\fill (2,1) circle (3pt);
	\fill (2,2) circle (3pt);
	\fill (2,3) circle (3pt);
	\fill (2,4) circle (3pt);
	\fill (2,5) circle (3pt);

	\node[left] at (-1,6) {{ \bf c): $\Delta_{can}$, $2 F(\Delta)_{can}$ }};

	\begin{scope}[xshift = -5cm, yshift = 4cm,rotate=-180]
	
	\draw[thick] (-1,3) -- (0,1);
	\draw[thick] (0,1) -- (-1,-1);
	\draw[thick] (-1,-1) -- (-2,1);
	\draw[thick] (-2,1) -- (-1,3);

	\draw[thin] (-1,2) -- (-3/2,1);
	\draw[thin] (-3/2,1) -- (-1,0);
	\draw[thin] (-1,0) -- (-1/2,1);
	\draw[thin] (-1/2,1) -- (-1,2);
	
	\fill (-1,0) circle (3pt) node[right] {$y_3$};
	\fill (-1,1) circle (3pt) node[right] {$y_2$};
	\fill (-1,2) circle (3pt) node[right] {$y_1$};

	\fill (-2,-1) circle (3pt);
	\fill (-2,0) circle (3pt);
	\fill (-2,1) circle (3pt);
	\fill (-2,2) circle (3pt);
	\fill (-2,3) circle (3pt);
	\fill (-1,-1) circle (3pt);
	\fill (-1,0) circle (3pt);
	\fill (-1,1) circle (3pt);
	\fill (-1,2) circle (3pt);
	\fill (-1,3) circle (3pt);
	\fill (0,-1) circle (3pt);
	\fill (0,0) circle (3pt);
	\fill (0,1) circle (3pt);
	\fill (0,2) circle (3pt);
	\fill (0,3) circle (3pt);

	\node [left] at (1,-2) {{\bf d): $\Delta_{can}$, $2 F(\Delta)_{can}$ }};

	\begin{scope}[xshift = 5cm, yshift = 3cm,rotate=180]

\draw[thick] (1,0) -- (2,1);
\draw[thick] (2,1) -- (1,4);
\draw[thick] (1,4) -- (0,1);
\draw[thick] (0,1) -- (1,0);

\draw[thin] (1,3) -- (3/2,3/2);
\draw[thin] (3/2,3/2) -- (1,1);
\draw[thin] (1,1) -- (1/2,3/2);
\draw[thin] (1/2,3/2) -- (1,3);

\fill (0,0) circle (3pt);
\fill (0,1) circle (3pt);
\fill (0,2) circle (3pt);
\fill (0,3) circle (3pt);
\fill (0,4) circle (3pt);
\fill (1,0) circle (3pt);
\fill (1,1) circle (3pt);
\fill (1,2) circle (3pt);
\fill (1,3) circle (3pt);
\fill (1,4) circle (3pt);
\fill (2,0) circle (3pt);
\fill (2,1) circle (3pt);
\fill (2,2) circle (3pt);
\fill (2,3) circle (3pt);
\fill (2,4) circle (3pt);	

rotate=0;
\node [left] at (-1,5) {{\bf e): $\Delta_{can}$, $2F(\Delta)_{can}$}};
	
	\fill (1,1) circle (3pt) node[right] {$y_1$};
	\fill (1,2) circle (3pt) node[right] {$y_2$};
	\fill (1,3) circle (3pt) node[right] {$y_3$};

\end{scope}
	\end{scope}

	\end{scope}
	\end{tikzpicture}
\end{center}	

This result suggests to consider all polytopes $\Delta$ among the $49$ with given $F(\Delta)$ at once. We found by observation:
\begin{proposition}
	In each of the $5$ classes there is exactly one maximal polytope with respect to the inclusion of sets.
\end{proposition}

\begin{remark} \label{remark_maximal_polytope_Delta_and_Delta_FI_same_normal_fan}
	\normalfont
	For $\Delta$ one of the maximal polytopes up to translation $\Delta$ coincides with $6 \cdot F(\Delta)$ in the cases $a)$ an $b)$ and with $4 \cdot  F(\Delta)$ in the cases $c), \, d)$ and $e)$.
\end{remark}
We picture the polytopes in the classes $a)$ and $b)$ in Figures \ref{figure_complete_list_polytopes_a} and \ref{figure_complete_list_polytopes_b}.

The following construction appeared first in (\cite[Cor.4.4]{Bat20}):

\begin{construction} \label{construction_with_the_many_points} 	\normalfont
	Let $\Delta \subset M_{\mathbb{R}}$ be a $3$-dimensional lattice polytope with Fine interior $F(\Delta) \neq \emptyset$. Let 
	\[ \tilde{\Delta}:= C(\Delta) + F(\Delta) \]
	be the Minkowski sum. In this way the normal fan $\Sigma_{\tilde{\Delta}}$ gets the coarsest refinement of $\Sigma_{C(\Delta)}$ and $\Sigma_{F(\Delta)}$, even if $F(\Delta)$ is not full dimensional (\cite[Prop.6.2.13]{CLS11}) and thus we get morphisms
	
	\begin{equation}  \label{diagram_toric_morphisms}
	\begin{tikzcd}
	& \mathbb{P}_{\tilde{\Delta}} \arrow[swap]{dl}{\rho} \arrow{dr}{\theta} \\
	\mathbb{P}_{C(\Delta)} && \mathbb{P}_{F(\Delta)}
	\end{tikzcd}
	\end{equation}
	where $\rho$ is birational and $\theta$ is birational if $\dim \, F(\Delta) = 3$. We prove by observation

\end{construction}

\begin{proposition} \label{proposition_iso_in_codimension_one}
	For $\Delta$ out of the $49$ polytopes $\Sigma_{C(\Delta)}[1] = \Sigma_{\tilde{\Delta}}[1]$, i.e. $\rho: \mathbb{P}_{\tilde{\Delta}} \rightarrow \mathbb{P}_{C(\Delta)}$ is an isomorphism in codimension $1$ and thus $Z_{C(\Delta)} \cong Z_{\tilde{\Delta}}$.
\end{proposition}

\begin{proposition} \label{proposition_maximal_polytopes_canonical_models}
	For $\Delta$ one of the $5$ maximal lattice polytopes among the $49$ polytopes we have $\Delta = C(\Delta)$ and
	\[ \mathbb{P}_{F(\Delta)} \cong \mathbb{P}_{\tilde{\Delta}} \cong \mathbb{P}_{C(\Delta)}. \]
\end{proposition}
\begin{proof}
	The reason for the latter is that by Remark \ref{remark_maximal_polytope_Delta_and_Delta_FI_same_normal_fan} the polytopes $\Delta$ and $F(\Delta)$ have the same normal fan.
\end{proof}

\begin{theorem} \label{proposition_to_toric_resolution_diagram_morphism}
	\cite[Thm.4.3]{Bat20} \\
	We have $\Sigma_{\tilde{\Delta}}[1] \subset S_{F}(\Delta)$. 
	In particular we may choose a simplicial fan $\Sigma$ with $\Sigma[1] = S_{F}(\Delta)$ refining $\Sigma_{\tilde{\Delta}}$.
\end{theorem}

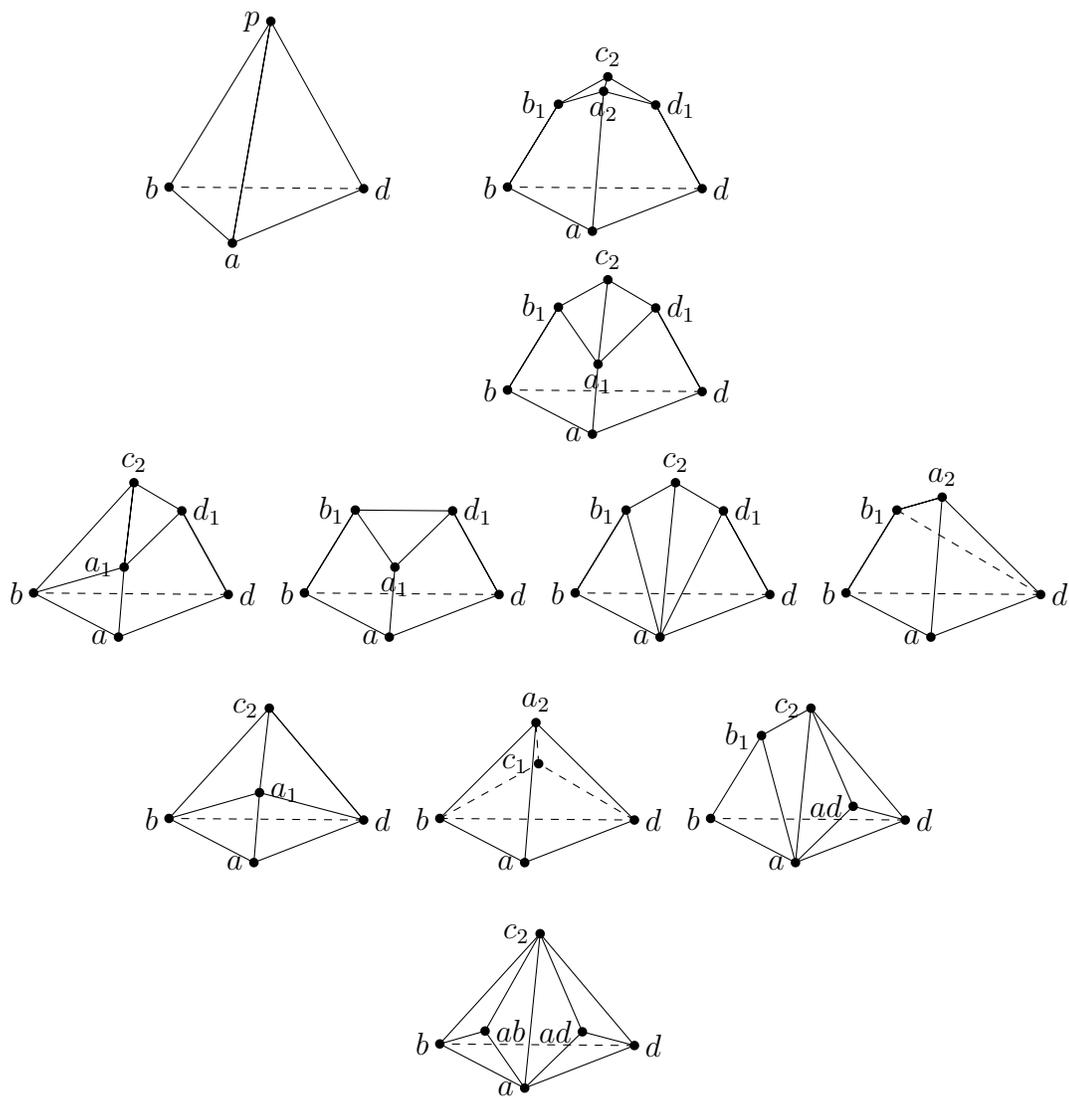
\begin{figure}

	\begin{tikzpicture}[scale = 0.6,rotate=0]

	\begin{scope}[scale = 1,yshift = 0cm, xshift = 5cm, rotate = -16]
	
	\draw (3/2,-1,5/2) -- (3,0,0);
	\draw (0,0,3) -- (0,3,0) -- (3/2,-1,5/2) -- (0,3,0) -- (3,0,0);
	\draw (0,0,3) -- (3/2,-1,5/2);
	
	\draw[dashed] (3,0,0) -- (0,0,3);
	

	\fill (3/2,-1,5/2) circle (3pt) node[below] {$a$};
	\fill (3,0,0) circle (3pt) node[right] {$d$};
	\fill (0,3,0) circle (3pt) node[left] {$p$};
	\fill (0,0,3) circle (3pt) node[left] {$b$};
	
	
	

	\end{scope}

	\begin{scope}[xshift = 12.5cm, yshift = 0cm, rotate = -16]
	

	\draw (3/2,-1,3/2) -- (0,0,3);
	\draw (3,0,0) -- (3/2,-1,3/2) -- (1/2,5/3,1/2) -- (0,1.5,1.5) -- (0,0,3) -- (0,1.5,1.5) -- (1/2,2,1/2) -- (1.5,1.5,0) -- (3,0,0) -- (1.5,1.5,0) -- (1/2,5/3,1/2) -- (1/2,2,1/2);
	
	\draw[dashed] (3,0,0) -- (0,0,3);
	
	
	\fill (3,0,0) circle (3pt) node[right] {$d$};
	\fill (3/2,-1,3/2) circle (3pt) node[left] {$a$};
	\fill (0,0,3) circle (3pt) node[left] {$b$};
	\fill (1/2,5/3,1/2) circle (3pt) node[below] {$a_2$};
	\fill (1.5,1.5,0) circle (3pt) node[right] {$d_1$};
	\fill (0,1.5,1.5) circle (3pt) node[left] {$b_1$};
	\fill (1/2,2,1/2) circle (3pt) node[above] {$c_2$};	
	

	\end{scope}

	\begin{scope}[xshift = 12.5cm, yshift = -4.5cm, rotate = -16]
	

	\draw (3/2,-1,3/2) -- (0,0,3);
	\draw (3,0,0) -- (3/2,-1,3/2) -- (1,1/3,1) -- (0,1.5,1.5) -- (0,0,3) -- (0,1.5,1.5) -- (1/2,2,1/2) -- (1.5,1.5,0) -- (3,0,0) -- (1.5,1.5,0) -- (1,1/3,1) -- (1/2,2,1/2);
	
	\draw[dashed] (3,0,0) -- (0,0,3);
	
	
	\fill (3,0,0) circle (3pt) node[right] {$d$};
	\fill (3/2,-1,3/2) circle (3pt) node[left] {$a$};
	\fill (0,0,3) circle (3pt) node[left] {$b$};
	\fill (1,1/3,1) circle (3pt) node[below] {$a_1$};
	\fill (1.5,1.5,0) circle (3pt) node[right] {$d_1$};
	\fill (0,1.5,1.5) circle (3pt) node[left] {$b_1$};
	\fill (1/2,2,1/2) circle (3pt) node[above] {$c_2$};	
	
	
	\end{scope}

	\begin{scope}[xshift = 2cm, yshift = -9cm, rotate = -16]

	\draw (3/2,-1,3/2) -- (0,0,3);
	\draw (3,0,0) -- (3/2,-1,3/2) -- (1,1/3,1) -- (1.5,1.5,0) -- (3,0,0) -- (1.5,1.5,0)  -- (1/2,2,1/2) -- (1,1/3,1) -- (1/2,2,1/2) -- (0,0,3) -- (1,1/3,1);
	
	\draw[dashed] (3,0,0) -- (0,0,3);
	

	\fill (3,0,0) circle (3pt) node[right] {$d$};
	\fill (3/2,-1,3/2) circle (3pt) node[left] {$a$};
	\fill (0,0,3) circle (3pt) node[left] {$b$};
	\fill (1,1/3,1) circle (3pt) node[left] {$a_1$};
	\fill (1.5,1.5,0) circle (3pt) node[right] {$d_1$};
	\fill (1/2,2,1/2) circle (3pt) node[above] {$c_2$};	
	
	
	\end{scope}

	\begin{scope}[xshift = 8cm, yshift = -9cm, rotate = -16]

	\draw (3/2,-1,3/2) -- (0,0,3);
	\draw (3,0,0) -- (3/2,-1,3/2) -- (1,1/3,1) -- (0,1.5,1.5) -- (0,0,3) -- (0,1.5,1.5) -- (1.5,1.5,0) -- (3,0,0) -- (1.5,1.5,0) -- (1,1/3,1);
	
	\draw[dashed] (3,0,0) -- (0,0,3);
	

	\fill (3,0,0) circle (3pt) node[right] {$d$};
	\fill (3/2,-1,3/2) circle (3pt) node[left] {$a$};
	\fill (0,0,3) circle (3pt) node[left] {$b$};
	\fill (1,1/3,1) circle (3pt) node[below] {$a_1$};
	\fill (1.5,1.5,0) circle (3pt) node[right] {$d_1$};
	\fill (0,1.5,1.5) circle (3pt) node[left] {$b_1$};
	
	
	\end{scope}

	\begin{scope}[xshift = 14cm, yshift = -9cm, rotate = -16]

	\draw (3/2,-1,3/2) -- (0,0,3);
	\draw (3,0,0) -- (3/2,-1,3/2) -- (0,1.5,1.5) -- (0,0,3) -- (0,1.5,1.5) -- (1/2,2,1/2) -- (1.5,1.5,0) -- (3,0,0) -- (1.5,1.5,0) -- (3/2,-1,3/2) -- (1/2,2,1/2);
	
	\draw[dashed] (3,0,0) -- (0,0,3);
	

	\fill (3,0,0) circle (3pt) node[right] {$d$};
	\fill (3/2,-1,3/2) circle (3pt) node[left] {$a$};
	\fill (0,0,3) circle (3pt) node[left] {$b$};
	\fill (1.5,1.5,0) circle (3pt) node[right] {$d_1$};
	\fill (0,1.5,1.5) circle (3pt) node[left] {$b_1$};
	\fill (1/2,2,1/2) circle (3pt) node[above] {$c_2$};	
	
	
	\end{scope}

	\begin{scope}[xshift = 20cm, yshift = -9cm, rotate = -16]

	\draw (3/2,-1,3/2) -- (0,0,3);
	\draw (3,0,0) -- (3/2,-1,3/2) -- (1/2,5/3,1/2) -- (0,3/2,3/2) -- (0,0,3) -- (0,3/2,3/2)-- (1/2,5/3,1/2) -- (3,0,0);
	
	\draw[dashed] (3,0,0) -- (0,3/2,3/2);
	\draw[dashed] (3,0,0) -- (0,0,3);
	

	\fill (3,0,0) circle (3pt) node[right] {$d$};
	\fill (3/2,-1,3/2) circle (3pt) node[left] {$a$};
	\fill (0,0,3) circle (3pt) node[left] {$b$};
	\fill (1/2,5/3,1/2) circle (3pt) node[above] {$a_2$};
	\fill (0,1.5,1.5) circle (3pt) node[left] {$b_1$};
	
	\end{scope}

	\begin{scope}[xshift = 5cm, yshift = -14cm, rotate = -16]

	\draw (3,0,0) -- (1/2,2,1/2) -- (0,0,3) -- (3/2,-1,3/2) -- (1,1/3,1) -- (3,0,0);
	\draw (0,0,3) -- (1,1/3,1) --(1/2,2,1/2) -- (3,0,0) -- (3/2,-1,3/2);
	\draw[dashed] (3,0,0) -- (0,0,3);
	

	\fill (3,0,0) circle (3pt) node[right] {$d$};
	\fill (3/2,-1,3/2) circle (3pt) node[left] {$a$};
	\fill (0,0,3) circle (3pt) node[left] {$b$};
	\fill (1,1/3,1) circle (3pt) node[right] {$a_1$};
	\fill (1/2,2,1/2) circle (3pt) node[left] {$c_2$};
	
	
	\end{scope}

	\begin{scope}[xshift = 11cm, yshift = -14cm, rotate = -16]
	
	\draw[dashed] (0,0,3) -- (3,0,0);
	\draw[dashed] (0,0,3) -- (1,1,1);
	\draw[dashed] (3,0,0) -- (1,1,1);
	\draw[dashed] (1,1,1) -- (1/2,5/3,1/2);
	\draw (3/2,-1,3/2) -- (0,0,3);
	\draw (3,0,0) -- (1/2,5/3,1/2);
	\draw (1/2,5/3,1/2) -- (3/2,-1,3/2);
	\draw (3/2,-1,3/2) -- (3,0,0);
	\draw (0,0,3) -- (1/2,5/3,1/2);
	

	\fill (3,0,0) circle (3pt) node[right] {$d$};
	\fill (3/2,-1,3/2) circle (3pt) node[left] {$a$};
	\fill (0,0,3) circle (3pt) node[left] {$b$};
	\fill (1/2,5/3,1/2) circle (3pt) node[above] {$a_2$};
	\fill (1,1,1) circle (3pt) node[left] {$c_1$};
	
	
	\end{scope}

	\begin{scope}[xshift = 17cm, yshift = -14cm, rotate = -16]
	
	\draw (3/2,-1,3/2) -- (0,0,3);
	\draw (3/2,-1,3/2) -- (1/2,2,1/2);
	\draw (3,0,0) -- (1/2,2,1/2);
	\draw (3,0,0) -- (2,1/6,1/2);
	\draw (3/2,-1,3/2) -- (2,1/6,1/2);
	\draw (1/2,2,1/2) -- (2,1/6,1/2);
	\draw (3/2,-1,3/2) -- (0,3/2,3/2);
	\draw (0,3/2,3/2) -- (1/2,2,1/2);
	\draw (3/2,-1,3/2) -- (3,0,0);
	\draw[dashed] (3,0,0) -- (0,0,3);
	\draw (0,3/2,3/2) -- (0,0,3);

	\fill (3,0,0) circle (3pt) node[right] {$d$};
	\fill (3/2,-1,3/2) circle (3pt) node[left] {$a$};
	\fill (0,0,3) circle (3pt) node[left] {$b$};
	\fill (0,3/2,3/2) circle (3pt) node[left] {$b_1$};
	\fill (1/2,2,1/2) circle (3pt) node[left] {$c_2$};
	\fill (2,1/6,1/2) circle (3pt) node[left] {$ad$};
	
	\end{scope}

	\begin{scope}[xshift = 11cm, yshift = -19cm, rotate = -16]
	
	\draw (3/2,-1,3/2) -- (0,0,3);
	\draw (3/2,-1,3/2) -- (1/2,2,1/2);
	\draw (3,0,0) -- (1/2,2,1/2);
	\draw (3,0,0) -- (2,1/6,1/2);
	\draw (1/2,2,1/2) -- (2,1/6,1/2);
	\draw (3/2,-1,3/2) -- (2,1/6,1/2);
	\draw (3/2,-1,3/2) -- (1/2,1/6,2);
	\draw (1/2,1/6,2) -- (1/2,2,1/2);
	\draw (3/2,-1,3/2) -- (3,0,0);
	\draw[dashed] (3,0,0) -- (0,0,3);
	\draw (1/2,1/6,2) -- (0,0,3);
	\draw (0,0,3) -- (1/2,2,1/2);

	\fill (3,0,0) circle (3pt) node[right] {$d$};
	\fill (3/2,-1,3/2) circle (3pt) node[left] {$a$};
	\fill (0,0,3) circle (3pt) node[left] {$b$};
	\fill (1/2,1/6,2) circle (3pt) node[right] {$ab$};
	\fill (1/2,2,1/2) circle (3pt) node[left] {$c_2$};
	\fill (2,1/6,1/2) circle (3pt) node[left] {$ad$};
	
	\end{scope}

	\end{tikzpicture}
	\caption{The 11 canonically closed polytopes out of 20 polytopes in the first class a). The polytopes are ordered in rows descendingly by their number of lattice points (The maximal polytope is additionaly put in the first row on the left).} \label{figure_complete_list_polytopes_a}
\end{figure}

\input{Pictures_2}

	\section{Construction of some minimal/canonical surfaces of general type in toric 3-folds}
\label{chapter_construction_minimal_canonical_models}
	\subsection{A toric framwork for constructing minimal and canonical models of hypersurfaces} \label{section_Fine_interior}
	 In this section we summarize some results from (\cite{Bat20}). The results of this section stay true in any dimension. In order to speak about minimal (canonical) models, we should give here a Definition of terminal (canonical) singularities of algebraic varieties. But we will just need this notions for toric varieties (in fact toric $3$-folds) and algebraic surfaces. In these two cases the situation could be simplified:
	\begin{definition} \label{definition_canonical_terminal_singularities}
		\cite[(1.11)]{Rei83} \\
		Consider a fan $\Sigma$ in $N_{\mathbb{R}}$ and a cone $\sigma$ of $\Sigma$. Then $\sigma$ is called canonical (of index $j \in \mathbb{N}$), if there exists a primitive vector $m \in M$ such that
		\[ \langle m, \nu_i \rangle = j \quad \textrm{for } \nu_i \in \sigma[1] \]
		and 
		\begin{align} \label{inequality_canonical_toric_variety}
		\langle m, n \rangle \geq j \quad  \textrm{for } n \in \sigma \cap N, \, n \notin \{0\} \cup \sigma[1]
		\end{align}
		$\sigma$ is called terminal if we have strict inequality in (\ref{inequality_canonical_toric_variety}). The fan $\Sigma$ is called canonical (terminal) if all its cones are canonical (terminal).
	\end{definition}

	According to \cite[(1.12)]{Rei83} we have the following result:
	\begin{theorem}
		The toric variety $\mathbb{P}_{\Sigma}$ has at most canonical (terminal) singularities if and only if the fan $\Sigma$ is canonical (terminal).
	\end{theorem}

	For normal algebraic surfaces we have
	\begin{proposition} (\cite[Thm.4.5]{KM98}) \\
		A normal algebraic surface $Y$ has canonical (terminal) singularities if and only if it has at most rational double points (is smooth). We choose the abbreviation R.D.P. for rational double points.
	\end{proposition}

	\begin{definition}
	A (partial) resolution of singularities $\pi: X \rightarrow Y$ of normal varieties is called \textit{crepant} if
	 \[ K_X = \pi^{*}(K_Y). \]
	Given a normal projective surface $Y$, we write $\kappa(Y)$ for the Kodaira dimension of $Y$, which by Definition is the Kodaira dimension of a resolution of singularities of $Y$. \\
	Given a normal algebraic surface $Y$ with $\kappa(Y) \geq 0$ a smooth projective surface $Y_{min}$ birational to $Y$ and with $K_{Y_{min}}$ nef is called a minimal model of $Y$. \\
	If $\kappa(Y) = 2$ a projective surface $Y_{can}$ birational to $Y$ with at most R.D.P. and with $K_{Y_{can}}$ ample is called a canonical model of $Y$.	
	\end{definition}

	\begin{remark}
	\normalfont
	If $Y$ is a surface with $\kappa(Y) \geq 0$ there exists up to isomorphism a unique minimal model of $Y$ and if $\kappa(Y) = 2$ there exists up to isomorphism a unique canonical model of $Y$. In the latter case there is a birational morphism $\phi: Y_{min} \rightarrow Y_{can}$ from the minimal model $Y_{min}$ to the canonical model $Y_{can}$, given by contracting every curve $C$ with $K_Y.C = 0$ (\cite[Ch.7, Cor.2.3]{BHPV04}). By a minimal or canonical model of $Z_f$ we mean a minimal or canonical model of $Z_{\Delta}$.
	\end{remark}

	\begin{theorem} \label{theorem_Kodaira_dimension_dimension_Fine_interior} (\cite[Thm.6.2]{Bat20}) \\
		Given a $3$-dimensional lattice polytope $\Delta \subset M_{\mathbb{R}}$ with $k:= \dim(F(\Delta)) \geq 0$, the Kodaira dimension of $Z_{\tilde{\Delta}}$ equals
		\[ \kappa(Z_{\tilde{\Delta}}) = \min(k,2). \]
	\end{theorem}

	For nondegenerate toric hypersurfaces these birational models could be constructed explicitly:
	\begin{theorem}
	The toric variety $\mathbb{P}_{\tilde{\Delta}}$ has at most canonical singularities and $\mathbb{P}_{\Sigma}$ has at most terminal singularities.
	\end{theorem}

	\begin{theorem} (\cite[Thm.5.4]{Bat20}, \cite[Thm.6.3]{Gie21}) \\
		The closure $Z_{\Sigma}$ of $Z_f$ gets a minimal model of $Z_f$ and if
		\[ \dim \, \Delta = \dim \, F(\Delta) = 3 \]
		the closure $Z_{F(\Delta)}$ of $Z_f$ gets a canonical model of $Z_f$ and the morphism $Z_{\Sigma} \rightarrow Z_{F(\Delta)}$  coincides with the birational morphism $\phi$ from the minimal to the canonical model.
	\end{theorem}

	\begin{remark}
		\normalfont
	Thus for our $49$ polytopes $\Delta$ with $\textrm{dim } F(\Delta) = 3$ we have $\kappa(Z_{\tilde{\Delta}}) = 2$.
	By properties of canonical models of surfaces of general type the morphism $Z_{\Sigma} \rightarrow Z_{F(\Delta)}$ is crepant. 
	\end{remark}

\subsection{The singularities of \texorpdfstring{$Z_{\tilde{\Delta}}$}{x} and \texorpdfstring{$Z_{F(\Delta)}$}{x} for $\Delta$ canonically closed} \label{section_singularities_of_canonical_model}

\begin{construction} (The singularities of $Z_{\tilde{\Delta}}$) \\
	\normalfont
The hypersurface $Z_{\tilde{\Delta}}$ has, just as the other birational models, only isolated singularities and, assuming $\Delta = C(\Delta)$, these come from the one-dimensional singular locus of $\mathbb{P}_{\tilde{\Delta}}$ intersected with $Z_{\tilde{\Delta}}$. For this we use that if $\Delta = C(\Delta)$ by nondegeneracy $Z_{\tilde{\Delta}}$ does not pass through the torus fixed points of $\mathbb{P}_{\tilde{\Delta}}$ (\cite[Prop.5.1.3]{Tre10}). \\
Let $\tau \in \Sigma_{\tilde{\Delta}}[2]$ with associated orbit $\mathcal{O}_{\tau} \cong \mathbb{C}^{*}$ be such that $\mathcal{O}_{\tau}$ is contained in the singular locus of $\mathbb{P}_{\tilde{\Delta}}$. Locally around $x \in \mathcal{O}_{\tau} \cap Z_{\tilde{\Delta}}$ the toric variety $\mathbb{P}_{\tilde{\Delta}}$ looks like $\mathcal{O}_{\tau} \times P$ where $P$ is a normal $2$-dimensional toric variety (compare (\cite{Bat94})). \\
Then since $Z_{\tilde{\Delta}}$ intersects $\mathcal{O}_{\tau}$ transversally, it has the same singularity at $x$ as $P$. But a normal toric variety has at most $A_k$-singularities (\cite[Ch.10]{CLS11}). Note that by (\cite{Bat94}) $P$ is gotten by using the minimal $2$-dimensional sublattice $N(\tau) \subset N$ containing $\tau \cap N$ and taking the affine toric variety to the cone $\tau \subset N(\tau)_{\mathbb{R}}$. Concretely if $\tau$ has generators $\nu_i$ and $\nu_j$, then writing
\[ \nu_i - \nu_j = k \cdot (\textrm{primitive vector}) \]
$P$ has a singularity of type $A_{k-1}$ at $x \in \mathcal{O}_{\tau}$. If $\tau$ corresponds to an edge $\Gamma = \langle s,t \rangle$ of $C(\Delta)$ then $Z_{\tilde{\Delta}}$ which by Proposition \ref{proposition_iso_in_codimension_one} is isomorphic to $Z_{C(\Delta)}$ intersects $\mathcal{O}_{\tau}$ in $l$ points, where
\[ s - t = l \cdot (\textrm{primitive vector}). \]
Else $\mathcal{O}_{\tau}$ contracts to a torus fixed point on $\mathbb{P}_{C(\Delta)}$ and $Z_{\tilde{\Delta}}$ does not intersect $\mathcal{O}_{\tau}$.

\end{construction}

\begin{construction} (The singularities of $Z_{F(\Delta)}$) \\
	\normalfont
	We make the observation that in all of our examples the rays of $\Sigma$ which do not belong to $\Sigma_{F(\Delta)}[1]$ lie on the boundary on in the interior of only one $3$-dimensional cone, which we call $\sigma$. \\
	This allows us to illustrate the situation with some pictures, by taking a cross-section through this cone $\sigma$ (see Figure \ref{figure_illustration_refinement}). Here between $\Sigma_{\tilde{\Delta}}$ and $\Sigma$ we picture only those rays which lie on an edge in the cross-section and skip those refining a $3$-dimensional cone of $\Sigma_{\tilde{\Delta}}$. For if $\Delta = C(\Delta)$ then $Z_{\tilde{\Delta}}$ is nondegenerate with respect to $\tilde{\Delta}$, thus it does not pass through the torus fixed points of $\mathbb{P}_{\tilde{\Delta}}$.
	\\
	Any additional ray $\rho$, that is any point in the cross-section, introduces a toric divisor $E_{\rho}$ on $\mathbb{P}_{\Sigma}$. This divisor $E_{\rho}$ intersects $Z_{\Sigma}$ in finitely many $(-2)$-curves and if $\rho$ refines the interior of $\sigma$, then $E_{\rho} \cap Z_{\Sigma}$ gets contracted to the torus fixed point $p$ corresponding to $\sigma$. By choosing a connected neighborhood $U$ of $p$ such that $p$ is a deformation retract of $U$ we get that $E_{\rho} \cap Z_{\Sigma}$ is a deformation retract of the preimage of $U$, thus it is connected and consists of just one $(-2)$-curve.
	\\
	In this way we see that the diagram with vertices the points in the interior of $\sigma$ in the Figures \ref{figure_class_a_polytopes_normal_cone_refinement} and \ref{figure_class_b_polytopes_normal_cone_refinement} constitutes the Dynkin diagram of the singularity of $Z_{F(\Delta)}$ at the torus fixed point to $\sigma$. In these figures we use the following notations. \\
	In $a)$: \\
	$ n_1 := n_{pab} = (-1,3,1), \, n_2:= n_{pad} = (2,-3,1), \, n_3 := n_{pbd} = (0,0,-1).$ \\
	In $b)$: \\
	$n_1 := n_{pab} = (2,1,-2), n_2 := n_{pbc} = (0,-1,-1) , n_3:= n_{pcd} = (-1,-1,1), \\
	n_4:= n_{pad} = (0,1,2)$, where for example $n_{pab}$ is the inner facet normal to the facet $\langle p,a,b \rangle$.

\end{construction}

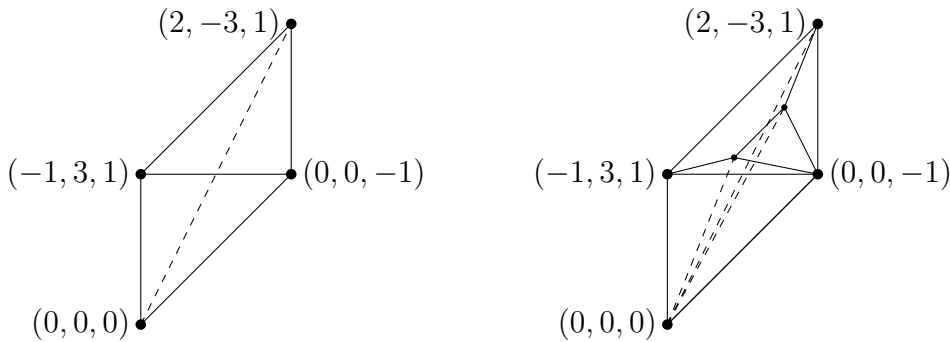
\begin{figure}[H]
	\begin{tikzpicture}[rotate=0, scale =1]
	
	\begin{scope}[xshift = 0cm]
	
	\fill (0,0) circle (2pt);
	\fill (2,2) circle (2pt);
	\fill (0,2) circle (2pt);
	\fill (2,4) circle (2pt);
	

	\node[left] at (0,2) {$(-1,3,1)$};
	\node[left] at (2,4) {$(2,-3,1)$};
	\node[right] at (2,2) {$(0,0,-1)$};
	\node[left] at (0,0) {$(0,0,0)$};

	\draw (0,0) -- (2,2);
	\draw[dashed] (0,0) -- (2,4);
	\draw (0,0) -- (0,2);
	\draw (2,2) -- (2,4);
	\draw (2,2) -- (0,2);
	\draw (2,4) -- (0,2);

	\end{scope}
	\begin{scope}[xshift = 7cm]
	
	\fill (0,0) circle (2pt);
	\fill (2,2) circle (2pt);
	\fill (0,2) circle (2pt);
	\fill (2,4) circle (2pt);
	
	\fill (2/9+2/3,14/9+2/3) circle (1.2pt);
	\fill (4/3+2/9, 8/9+2) circle (1.2pt);

	\node[left] at (0,2) {$(-1,3,1)$};
	\node[left] at (2,4) {$(2,-3,1)$};
	\node[right] at (2,2) {$(0,0,-1)$};
	\node[left] at (0,0) {$(0,0,0)$};
	
	\draw (0,0) -- (2,2) -- (2/9+2/3,14/9+2/3);
	\draw[dashed] (0,0) -- (2/9+2/3,14/9+2/3);
	\draw[dashed] (0,0) -- (4/3+2/9,8/9+2);
	\draw (2/9+2/3,14/9+2/3) -- (4/3+2/9,8/9+2);
	\draw (2,2) -- (4/3+2/9,8/9+2);
	
	\draw (2,4) -- (4/3+2/9,8/9+2);
	\draw (0,2) -- (2/9+2/3,14/9+2/3);

	\draw (0,0) -- (2,2);
	\draw[dashed] (0,0) -- (2,4);
	\draw (0,0) -- (0,2);
	\draw (2,2) -- (2,4);
	\draw (2,2) -- (0,2);
	\draw (2,4) -- (0,2);

	\end{scope}
	
	\end{tikzpicture}
	
	\caption{On the left $\sigma$ in class $a)$ is pictured and on the right a refinement of $\sigma$ as a cone is pictured, that is the added points are just rational multiples of the cone generators. This refinement shows $\Sigma_{\tilde{\Delta}}$.}
	\label{figure_illustration_refinement}
\end{figure}

\begin{example}
	\normalfont
	Consider the polytope $\langle a,b,d,b_1,c_2,d_1 \rangle$ in the class $a)$ (in the third row the third from the left): The orbit to the cone spanned by $n_1$ and $n_2$ corresponds to the edge $\langle a, d\rangle$ and does not come from an edge of $C(\Delta)$, therefore we omit the two lattice points between $n_1$ and $n_2$ in the Figure \ref{figure_class_a_polytopes_normal_cone_refinement}. The cones corresponding to the edges $\langle a,c_2 \rangle$, $\langle b_1, c_2 \rangle$, $\langle c_2, d_1 \rangle$ yield $A_2$ singularities. Since $a-c_2$, $b_1 -c_2$, $c_2 -d_1$ are primitive lattice vectors in total we get three $A_2$ singularities. On $Z_{F(\Delta)}$ we get one singularity of type $A_8$.
\end{example}

\begin{remark} \normalfont
	Note there are examples in the list where $Z_{F(\Delta)}$ has a singularity of type $D_5$, $D_7$ or $E_6$. If $\Delta$ is not canonically closed $\Delta$ has the same support vectors as its canonical closure (\cite[Cor.1.18]{Bat94}). Thus all support vectors lie again on the cone $\sigma$ and this might help in determining the singularities.

\end{remark}

\begin{figure}
	
	\begin{tikzpicture}
	
	\begin{scope}[xshift= 1cm, scale=0.7]
	
	\fill (0,0) circle(2pt);
	\fill (2,4) circle(2pt);
	\fill (4,0) circle(2pt);

	\draw (0,0) -- (4,0) -- (2,4) -- (0,0);
	
	\node[below] at (0,0) {$n_1$};
	\node[left] at (2,4) {$n_2$};
	\node[below] at (4,0) {$n_3$};
	
	\fill (2/3,4/3) circle(2pt);
	\fill (4/3,8/3) circle(2pt);
	
	
	
	\end{scope}

	\begin{scope}[xshift=5cm, scale=0.7]

	\fill (0,0) circle(2pt);
	\fill (2,4) circle(2pt);
	\fill (4,0) circle(2pt);
	
	\draw (0,0) -- (4,0) -- (2,4) -- (0,0);
	
	\node[below] at (0,0) {$n_1$};
	\node[left] at (2,4) {$n_2$};
	\node[below] at (4,0) {$n_3$};
	
	\fill (2/6+2,4/6) circle(2pt);
	\fill (2/3+2,4/3) circle(2pt);
	\draw (0,0) -- (2/6+2,4/6) -- (2/3+2,4/3) -- (2,4) -- (2/3+2,4/3) -- (4,0) -- (2/6+2,4/6);

	\fill (2/3,4/3) circle(2pt);	
	\fill (4/3,8/3) circle(2pt); 	
	\end{scope}

	\begin{scope}[yshift = -3.3cm, xshift=5cm, scale=0.7]

	\fill (0,0) circle(2pt);
	\fill (2,4) circle(2pt);
	\fill (4,0) circle(2pt);
	
	\draw (0,0) -- (4,0) -- (2,4) -- (0,0);
	
	\node[below] at (0,0) {$n_1$};
	\node[left] at (2,4) {$n_2$};
	\node[below] at (4,0) {$n_3$};
	
	\fill (4/9+4/3,8/9) circle(2pt);
	\fill (8/9+4/3,16/9) circle(2pt);
	\draw (0,0) -- (4/9+4/3,8/9) -- (8/9+4/3, 16/9) -- (2,4) -- (8/9+4/3, 16/9) -- (4,0) -- (4/9+4/3,8/9);

	\fill (2/6+2,4/6) circle(2pt);	
	\fill (2/3,4/3) circle(2pt); 	
	\fill (2/3+2,4/3) circle(2pt); 	
	\fill (4/3,8/3) circle(2pt); 	
	\fill (2/3+4/3,4/3) circle(2pt); 	
	
	\end{scope}

	\begin{scope}[yshift = -6.8cm, xshift=0cm,scale=0.7]
	
	\fill (0,0) circle(2pt);
	\fill (2,4) circle(2pt);
	\fill (4,0) circle(2pt);
	
	\draw (0,0) -- (4,0) -- (2,4) -- (0,0);
	
	\node[below] at (0,0) {$n_1$};
	\node[left] at (2,4) {$n_2$};
	\node[below] at (4,0) {$n_3$};
	
	\fill (4/9+4/3,8/9) circle(2pt);
	\fill (8/9+4/3,16/9) circle(2pt);
	
	\draw (2/9+4/3,4/9) -- (4,0) -- (8/9+4/3,16/9) -- (2,4) -- (8/9+4/3,16/9) -- (2/9+4/3,4/9) -- (0,0);

	\fill (2/3,4/3) circle(2pt); 	
	\fill (2/3+2,4/3) circle(2pt); 	
	\fill (4/3,8/3) circle(2pt); 	
	\fill (2/3+4/3,4/3) circle(2pt); 
	\fill (2/9+4/3,4/9) circle(2pt); 
	\end{scope}

	\begin{scope}[yshift = -6.8cm, xshift=3.5cm,scale=0.7]
	
	\fill (0,0) circle(2pt);
	\fill (2,4) circle(2pt);
	\fill (4,0) circle(2pt);
	
	\draw (0,0) -- (4,0) -- (2,4) -- (0,0);
	
	\node[below] at (0,0) {$n_1$};
	\node[left] at (2,4) {$n_2$};
	\node[below] at (4,0) {$n_3$};

	\draw (0,0) -- (16/7,8/7) -- (2,4) -- (16/7,8/7) -- (4,0);
	
	\fill (2/3,4/3) circle(2pt); 	
	\fill (4/3,8/3) circle(2pt); 	
	\fill (16/7, 8/7) circle(2pt); 
	\fill (8/10+2,8/10) circle(2pt); 
	\fill (8/15+8/5,16/15) circle(2pt);
	\fill (4/9+4/3,8/9) circle(2pt);
	\fill (10/15+8/5,20/15) circle(2pt);
	\fill (8/9+4/3,16/9) circle(2pt);
	
	\end{scope}

	\begin{scope}[yshift = -6.8cm, xshift=7cm,scale=0.7]
	
	\fill (0,0) circle(2pt);
	\fill (2,4) circle(2pt);
	\fill (4,0) circle(2pt);
	
	\draw (0,0) -- (4,0) -- (2,4) -- (0,0);
	
	\node[below] at (0,0) {$n_1$};
	\node[left] at (2,4) {$n_2$};
	\node[below] at (4,0) {$n_3$};

	\draw (4,0) -- (1/2+1,1) -- (2,2) -- (4,0) -- (1/2+1,1) -- (0,0) -- (1/2+1,1) -- (2,2) -- (2,4);
	
	\fill (3/2,1) circle(2pt); 
	\fill (2,2) circle(2pt); 

	\fill (2/6+2,4/6) circle(2pt);	
	\fill (2/3+2,4/3) circle(2pt); 	
	\fill (3/2+1/6,1+1/3) circle(2pt); 
	\fill (3/2+2/6,1+2/3) circle(2pt); 
	
	\fill (4/9+4/3,8/9) circle(2pt);
	\fill (8/9+4/3,16/9) circle(2pt);
	
	\end{scope}

	\begin{scope}[yshift = -6.8cm, xshift=10.5cm,scale=0.7]
	
	\fill (0,0) circle(2pt);
	\fill (2,4) circle(2pt);
	\fill (4,0) circle(2pt);
	
	\draw (0,0) -- (4,0) -- (2,4) -- (0,0);
	
	\node[below] at (0,0) {$n_1$};
	\node[left] at (2,4) {$n_2$};
	\node[below] at (4,0) {$n_3$};

	\draw (4,0) -- (4/9+8/3,8/9) -- (0,0) -- (4/9+8/3,8/9) -- (2,4);
	
	\fill (4/9+8/3, 8/9) circle(2pt); 

	\fill (2/6+2,4/6) circle(2pt);	
	\fill (2/3,4/3) circle(2pt); 	
	\fill (4/3,8/3) circle(2pt); 	
	
	\end{scope}

	\begin{scope}[yshift = -10.5cm, xshift=2cm,scale=0.7]
	
	\fill (0,0) circle(2pt);
	\fill (2,4) circle(2pt);
	\fill (4,0) circle(2pt);
	
	\draw (0,0) -- (4,0) -- (2,4) -- (0,0);
	
	\node[below] at (0,0) {$n_1$};
	\node[left] at (2,4) {$n_2$};
	\node[below] at (4,0) {$n_3$};

	\draw (4,0) -- (2/9+4/3,4/9) -- (10/9+4/3,20/9) -- (4,0) -- (2/9+4/3,4/9) -- (0,0) -- (2,4) -- (10/9+4/3,20/9);
	
	\fill (2/9+4/3,4/9) circle(2pt); 
	\fill (10/9+4/3, 20/9) circle(2pt); 

	\fill (4/9+4/3,8/9) circle(2pt);
	\fill (8/9+4/3,16/9) circle(2pt);
	\fill (2/3+4/3,4/3) circle(2pt); 
	
	\fill (2/3,4/3) circle(2pt); 	
	\fill (4/3,8/3) circle(2pt); 	
	
	\end{scope}

	\begin{scope}[yshift = -10.5cm, xshift=6cm,scale=0.7]
	
	\fill (0,0) circle(2pt);
	\fill (2,4) circle(2pt);
	\fill (4,0) circle(2pt);
	
	\draw (0,0) -- (4,0) -- (2,4) -- (0,0);
	
	\node[below] at (0,0) {$n_1$};
	\node[left] at (2,4) {$n_2$};
	\node[below] at (4,0) {$n_3$};

	\draw (4,0) -- (2/9+8/3,4/9) -- (4/9+8/3,8/9) -- (4,0) -- (2/9+8/3,4/9) -- (0,0) -- (2,4) -- (4/9+8/3,8/9);
	
	\fill (2/9+8/3,4/9) circle(2pt); 
	\fill (4/9+8/3, 8/9) circle(2pt); 

	\fill (2/3,4/3) circle(2pt); 	
	\fill (4/3,8/3) circle(2pt); 	
	\end{scope}

	\begin{scope}[yshift = -10.5cm, xshift=10cm,scale=0.7]
	
	\fill (0,0) circle(2pt);
	\fill (2,4) circle(2pt);
	\fill (4,0) circle(2pt);
	
	\draw (0,0) -- (4,0) -- (2,4) -- (0,0);
	
	\node[below] at (0,0) {$n_1$};
	\node[left] at (2,4) {$n_2$};
	\node[below] at (4,0) {$n_3$};

	\draw (4,0) -- (3/2,1) -- (7/6+1,7/3) -- (10/9+4/3,20/9) -- (4,0) -- (0,0) -- (3/2,1) -- (7/6+1,7/3) -- (2,4) -- (10/9+4/3,20/9);
	
	\fill (3/2,1) circle(2pt); 
	\fill (10/9+4/3, 20/9) circle(2pt); 
	\fill (7/6+1, 7/3) circle(2pt); 
	
	\fill (3/8+21/24+3/4,1/4+21/12) circle(2pt);
	\fill (6/8+14/24+2/4,2/4+14/12) circle(2pt);
	\fill (9/8+7/24+1/4,3/4+7/12) circle(2pt);	
	
	\fill (2/6+2,4/6) circle(2pt);
	\fill (4/9+4/3,8/9) circle(2pt);

	\end{scope}

	\begin{scope}[yshift = -14cm, xshift=5cm,scale=0.7]
	
	\fill (0,0) circle(2pt);
	\fill (2,4) circle(2pt);
	\fill (4,0) circle(2pt);
	
	\draw (0,0) -- (4,0) -- (2,4) -- (0,0);
	
	\node[below] at (0,0) {$n_1$};
	\node[left] at (2,4) {$n_2$};
	\node[below] at (4,0) {$n_3$};

	\draw (4,0) -- (2/9+4/3,4/9) -- (4/12+1,8/12) -- (7/6+1,7/3) -- (10/9+4/3,20/9) -- (4,0) -- (0,0) -- (2/9+4/3,4/9) -- (4/12+1,8/12) -- (0,0) -- (2,4) -- (7/6+1,7/3) -- (10/9+4/3,20/9) -- (2,4);

	\fill (10/9+4/3, 20/9) circle(2pt); 
	\fill (7/6+1, 7/3) circle(2pt); 

	\fill (2/9+4/3,4/9) circle(2pt); 
	\fill (4/12+1, 8/12) circle(2pt); 
	
	\fill (7/30+1/5+16/60+4/5,7/15+32/60) circle(2pt);
	\fill (14/30+2/5+12/60+3/5,14/15+24/60) circle(2pt);
	\fill (21/30+3/5+8/60+2/5,21/15+16/60) circle(2pt);
	\fill (28/30+4/5+4/60+1/5,28/15+8/60) circle(2pt);
	
	
	\end{scope}

	\end{tikzpicture}
	
	\caption{A cross-section of the cone $\sigma$ for all canonically closed polytopes in $a)$ with the refinement $\Sigma_{\tilde{\Delta}}$ and the position of the additional vectors from $S_{F}(\Delta)$ lying on an edge in the cross-section and yielding singularities of $Z_{\tilde{\Delta}}$. The order is the same as in Figure \ref{figure_complete_list_polytopes_a}.} \label{figure_class_a_polytopes_normal_cone_refinement}
\end{figure}
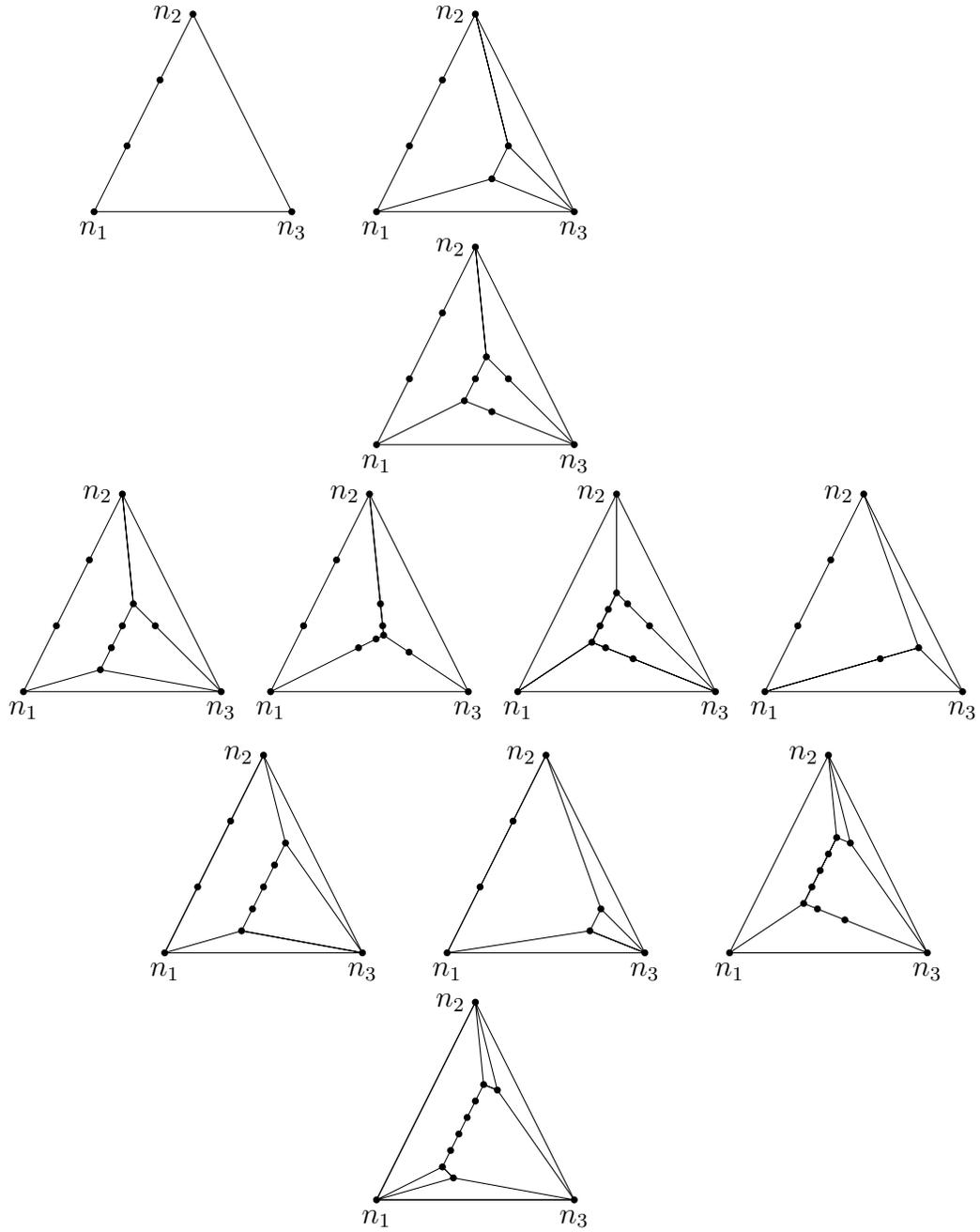

\input{Pictures_5}

\section{Kanev and Todorov surfaces}
\label{chapter_constructing_Kanev_Todorov_surfaces}

		\underline{Notation:} Let $Y$ be a complex algebraic surface with at most R.D.P., let
	\begin{itemize}
		\setlength\itemsep{0em}
		\item $p_g(Y) := h^{0}(Y, \mathcal{O}_Y(K_Y))$ the geometric genus of $Y$
		\item $q(Y) := h^{0}(Y, \Omega_{Y}^1)$ the irregularity of $Y$ 
		\item $\chi(Y, \mathcal{O}_Y)$: The euler characteristic of the structure sheaf
		\item $e(Y)$: The topological euler number.
	\end{itemize}
	
	\begin{construction} \label{constr_geom_genus_Kanev_Todorov_surfaces}
		\normalfont
	By (\cite[Thm.4.2, Prop.5.13]{Gie21}) we have for our $49$ examples
	\[ p_g(Z_{\Sigma}) = l(F(\Delta)) = l^*(\Delta) = 1, \]
	and $q(Z_{\Sigma}) = 0$. In (\cite[Appendix 3]{Sch18}) the euler number of $Y:= Z_{\Sigma}$ has already been computed for all of the $49$ polytopes. Namely $e(Y) = 23$ in $a)$ and $b)$ and $e(Y) = 22$ in $c)$, $d)$ and $e)$. Then we may deduce from Noether's formula (\cite[Ch.1, Thm.(5.5)]{BHPV04})
	\[ 2 = 1 - q(Y) + p_g(Y)  = \chi(Y, \mathcal{O}_{Y}) = \frac{1}{12}(K_{Y}^{2} + e(Y)) \]
	that $K_{Y}^2 = 1$ or $K_{Y}^2 = 2$. Minimal surfaces $Y$ with
	\[ p_{g}(Y) = 1, \quad q(Y) = 0, \quad  K_Y^2 = 1, \]
	are known as \textit{Kanev surfaces} (compare (\cite{Cat78})), whereas minimal surfaces $Y$ with
	\[ p_{g}(Y) = 1, \quad q(Y) = 0, \quad  K_Y^2 = 2, \]
	are known as \textit{Todorov surfaces.} 
	\end{construction}
	We obtain
	\begin{theorem}
		Let $\Delta$ be a polytope out of the $49$ examples. Then $Y:= Z_{\Sigma}$ gets a Kanev surface in $a)$ and $b)$ and a Todorov surface in $c)$, $d)$ and $e)$.
	\end{theorem}
	\qed

\begin{remark}
	\normalfont
	Note that $K_Y^2 =2$ if and only if $l^*(\Delta_{can}) = 3$ and else $K_Y^2 = 1$ and $l^*(\Delta_{can}) = 2$. This is no accident since by the adjunction formula
	\[ 2p_g(K_{Z_{\Delta}}) - 2 = 2 \cdot K_{Z_{\Delta}}^2 \Rightarrow K_{Z_{\Delta}}^2 = p_g(K_{Z_{\Delta}}) - 1. \]
	Note that by Construction \ref{specializing_facts_to_canonical_Fano_3_topes}, at least if $\Delta = C(\Delta)$, and by (\cite[Thm.4.2]{Gie21}) we have
	\[  p_g(K_{Z_{\Delta}}) = l^*(\Delta_{can}). \]
	But it seems to be tedious to deduce the analogous formula for $K_{Z_{\Sigma}}$ from this.
	
\end{remark}

	\subsection{Relating properties of Kanev/Todorov surfaces in general and as toric hypersurfaces}
	
	By (\cite[section 8]{Bomb73}) a Kanev surface $Y$ is simply connected and the condition $q(Y) = 0$ in the definition is actually superfluous. We may also compute the Hodge numbers of $Y$: The only nonzero Hodge numbers are
	\[ h^{0,0} = h^{2,2} = 1, \quad h^{0,2} = h^{2,0} = 1, \quad h^{1,1} = 19 \]	
	Similarly for Todorov surfaces the only nonzero Hodge numbers are
	\[ h^{0,0} = h^{2,2} = 1, \quad  h^{2,0} = h^{0,2} = 1, \quad h^{1,1} = 18. \]
	The following result is well known (\cite[Ch.7, Cor.(5.4)]{BHPV04}).
	\begin{proposition} \label{proposition_plurigenera_surfaces_of_gen_type}
		Let $Y$ be a minimal or canonical surface of general type, then we have the following formula for the plurigenera
		\[ P_{n}(Y) := h^{0}(Y, \mathcal{O}_Y(nK_Y)) = \chi(Y, \mathcal{O}_Y) + \frac{n(n-1)}{2} K_Y^2 \quad \textrm{ for } n \geq 2 \]
	\end{proposition}

	\begin{remark}
	\normalfont
	In the case of toric hypersurfaces Theorem 4.2 in \cite{Gie21} does not only compute $p_g(Y)$ but all plurigenera $P_n(Y)$ for $n \geq 1$ by the formula
	\[ P_n(Y) = l(n \cdot F(\Delta)) - l^*((n-1)F(\Delta)).  \]

\end{remark}

	\begin{figure} 
		\begin{tikzpicture}[scale=0.9]
			
				\begin{scope}[scale = 1.5,yshift = 0cm, xshift = 0cm, rotate = -16]
			
			\draw (3/2,-1,5/2) -- (3,0,0);
			\draw (0,0,3) -- (0,3,0) -- (3/2,-1,5/2) -- (0,3,0) -- (3,0,0);
			\draw (0,0,3) -- (3/2,-1,5/2);
			
			\draw[dashed] (3,0,0) -- (0,0,3);
			
			\draw[dashed] (3/2,0,3/2) -- (0,3,0);
			\draw[dashed] (3/2,-1,5/2) -- (3/2,0,3/2);
			

			\fill (3/2,-1,5/2) circle (2pt) node[below] {$a$};
			\fill (3,0,0) circle (2pt) node[right] {$d$};
			\fill (0,3,0) circle (2pt) node[left] {$p$};
			\fill (0,0,3) circle (2pt) node[left] {$b$};

			\fill (1,1/3,5/2-5/6) circle (1.4pt) node[left] {$a_1$};
			\fill (1/2,5/3,5/2-10/6) circle(1.4pt) node[left] {$a_2$};
		
			\node at (3/2,0,3/2) {$\circ$};
			\node at (3/2+0.25,0+0.1,3/2+0.2) {$c$};
			\node at (3/6,2,3/6) {$\circ$};
			\node at (3/6+0.3,2+0.1,3/6+0.2) {$c_2$};
			\node at (1,1,1) {$\circ$};
			\node at (1.3,1.1,1.2) {$c_1$};
			\fill (3/2,3/2,0) circle(1.4pt) node[right] {$d_1$};
			\fill (0,3/2,3/2) circle(1.4pt) node[left] {$b_1$};
		
			\node at (3/2,-1/3,3/2+1/3) {$\circ$};
			\node at (3/2,-2/3,3/2+2/3) {$\circ$};
		
			\node at (1/2,1/2,2) {$\circ$};
			\node at (1/2-0.3,1/2-0.1,2-0.2) {$bc$};
			\node at (2,1/2,1/2) {$\circ$};
			\node at (2+0.3,1/2+0.1,1/2+0.2) {$cd$};

			\fill (1/2,1/6,5/6+3/2) circle (1.4pt) node[below] {$ab$};
			\fill (1/2+3/2,1/6,5/6) circle (1.4pt) node[below] {$ad$};
			
			\node at (1,2/3,8/6) {$\circ$};
			\node at (1+0.23, 2/3+0.1,8/6+0.2) {$0$};

			\end{scope}

		\begin{scope}[xshift = 7cm, yshift = -2cm, rotate= 0, scale =1.1]
		\draw (0,0,3) -- (3,0,0);
		\draw (3,0,0) -- (3,3,0);
		\draw[dashed] (3,3,0) -- (1,4,2);
		\draw[dashed] (1,4,2) -- (0,0,3);
		
		\draw (3,0,0) -- (-1,5,-4);
		\draw (0,0,3) -- (-1,5,-4);
		\draw (3,3,0) -- (-1,5,-4);
		\draw[dashed] (1,4,2) -- (-1,5,-4);
		
		\draw[dashed] (3,0,0) -- (1,4,2);
		\draw[dashed] (0,0,3) -- (3,3,0);
		
		\node at (1,1,2) {$\circ$};
		\node at (2,2,1) {$\circ$};
		
		\fill (-1,5,-4) circle (2.7pt) node[right] {$p$};
		\fill (0,0,3) circle (2.7pt) node[left] {$a$};
		\fill (3,0,0) circle (2.7pt) node[right] {$d$} ;
		\fill (3,3,0) circle (2.7pt) node[right] {$c$};
		
		\node at (1,4,2) {$\circ$};
		\node at (1 - 0.3,4 - 0.1,2 - 0.2) {$b$};
		
		\fill (-1/3,5/3,2/3) circle (1.8pt) node[left] {$a_1$};
		\fill (-2/3,10/3, -5/3) circle(1.8pt) node[left] {$a_2$};
		\fill (5/3,11/3,-4/3) circle(1.8pt) node[right] {$c_1$};
		\fill (3-8/3,3+4/3,-8/3) circle(1.8pt) node[right] {$c_2$};
		
		\node at (0,4+1/2,-1) {$\circ$};
		\node at (0.35,4+1/2+0.1,-1+0.2) {$b_1$};
		\fill (1,5/2,-2) circle(1.8pt) node[left] {$d_1$};
		
		\node at (1/2+5/6,2+11/6,1-2/3) {$\circ$};
		\node at (1/2+5/6-0.35, 2+ 11/6, 1- 2/3-0.2) {$bc$};
		
		\node at (1/3, 5/6+2,1/3+1) {$\circ$};
		\node at (1/3, 5/6+2+0.3,1/3+1) {$ab$};
		
		\fill (3/2-1/6,5/6,2/6) circle (1.8pt) node[left] {$ad$};
		\fill (3/2+5/6, 11/6,-2/3) circle (1.8pt) node[above] {$cd$};
		
		\node at (2/3,8/3,-1/3) {$\circ$};
		\node at (2/3+0.3, 8/3,-1/3+0.2) {$0$};
		
		\end{scope}

		\end{tikzpicture}
		\caption{The maximal polytope in the class $a)$ on the left and in the class $b)$ on the right. In both cases $l(\Delta) = 18$.} \label{figure_maximal_polytopes_a_and_b}
	\end{figure}
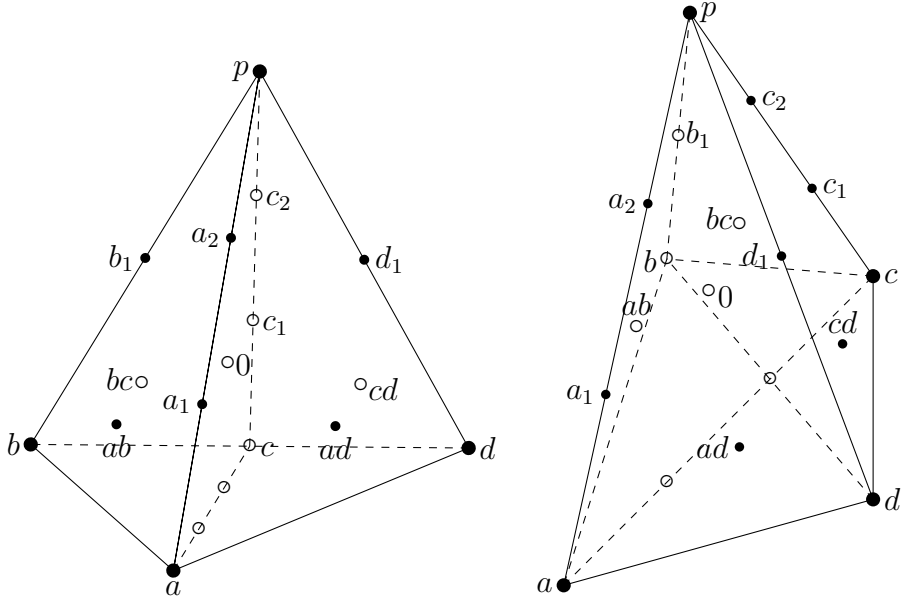

	

	\begin{example} (Kanev's original example) \label{kanev's_example_of_kanev_surface} \\
	\normalfont
	The first Kanev surface was found in (\cite{Kan76}) via the following Fermat polynomial
	\begin{align*} f(x_0,...,x_3) := x_0^6 + x_1^6 + x_2^6 + x_3^6,
	\end{align*}
	Let $X:= \{f=0\} \subset \mathbb{P}^3$, $G := \mathbb{Z}/2\mathbb{Z} \times \mathbb{Z}/3\mathbb{Z}$. Then $G$ acts on $\mathbb{P}^3$ via
	\[ g = (-1,\epsilon): (x_{0}:x_{1}:x_{2}:x_{3}) \mapsto (x_{0}:\epsilon x_{1}:\epsilon^{2} x_{2}:-x_{3}) \]
	where $\epsilon := e^{\frac{2\pi i}{3}}$, and $G$ leaves $X$ invariant. The quotient $X/G$ is a singular hypersurface in the toric variety $\mathbb{P}^{3}/G$. Letting $Y_{min}$ denote the minimal resolution of $X/G$, $Y_{min}$ gets a Kanev surface. 
	\\
	By \cite[Example 6.1]{BKS19} the hypersurface $(X/G) \cap T$ has Newton polytope:
	\[ \Delta' := \langle (1,2,4), (1,0,0), (1,4,2), (-2,-4,-5) \rangle \]
	This polytope is isomorphic to the maximal polytope in a). Note that in this example as well as in the example $c)$ the toric variety $\mathbb{P}_{\tilde{\Delta}}$ is almost a weighted projective space, more precisely a fake weighted projective space, that is $\rk \, \Cl(\mathbb{P}_{\tilde{\Delta}}) = 1$.
\end{example}

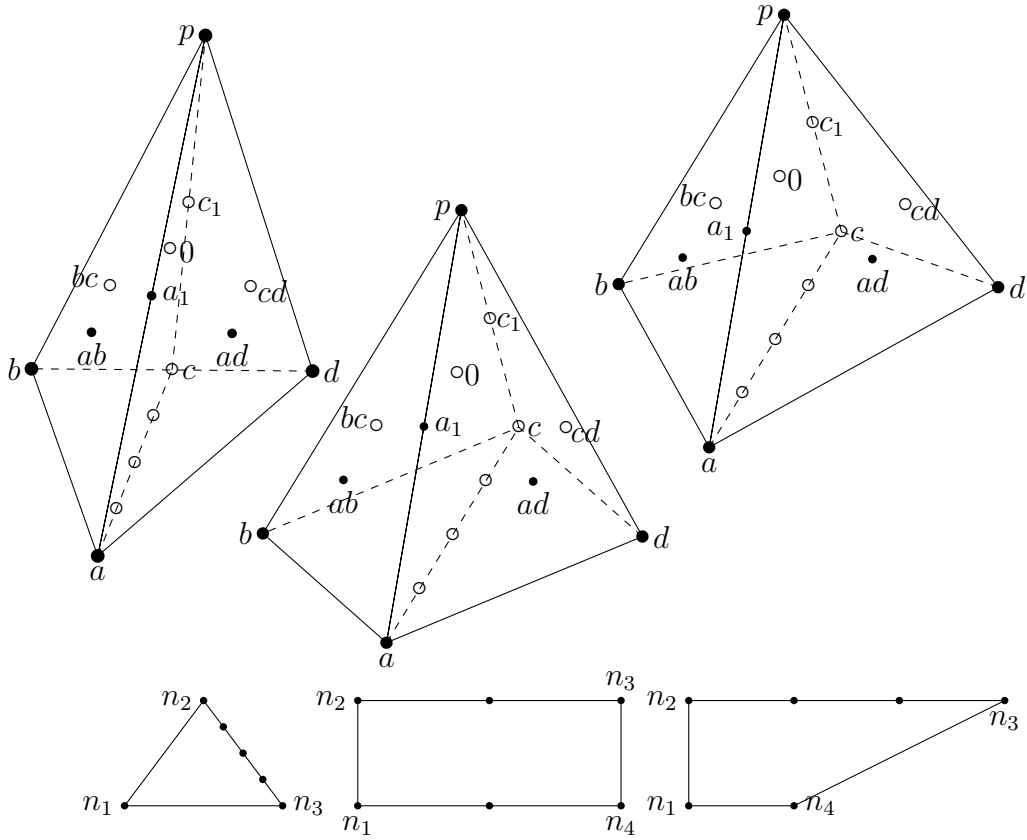
\begin{figure}
	\begin{tikzpicture}
		
	\begin{scope}[scale = 1.3,yshift = 0cm, xshift = 0cm, rotate = -16]
	
	\draw (3/2,-4/3,3/2+4/3) -- (2,0,0);
	\draw (0,0,2) -- (0,3,0) -- (3/2,-4/3,3/2+4/3) -- (0,3,0) -- (2,0,0);
	\draw (0,0,2) -- (3/2,-4/3,3/2+4/3);
	
	\draw[dashed] (2,0,0) -- (0,0,2);
	
	\draw[dashed] (1,0,1) -- (0,3,0);
	\draw[dashed] (3/2,-4/3,3/2+4/3) -- (1,0,1);
	

	\fill (3/2,-4/3,17/6) circle (2pt) node[below] {$a$};
	\fill (2,0,0) circle (2pt) node[right] {$d$};
	\fill (0,3,0) circle (2pt) node[left] {$p$};
	\fill (0,0,2) circle (2pt) node[left] {$b$};

	\fill (3/4,-4/3+13/6,17/12) circle (1.4pt) node[right] {$a_1$};
	
	\node at (1,0,1) {$\circ$};
	\node at (1+0.25,0+0.1,1+0.2) {$c$};
	\node at (1/2,3/2,1/2) {$\circ$};
	\node at (1/2+0.3,3/2+0.1,1/2 + 0.2) {$c_1$};

	\node at (1+1/8,-1/3,3/4 + 17/24) {$\circ$};
	\node at (1+2/8,-2/3,2/4+17/12) {$\circ$};
	\node at(1+3/8, -3/3,1/4+51/24) {$\circ$};
	
	\node at (1/4,3/4,1/4+1) {$\circ$};
	\node at (1/4-0.2,3/4+0.1,1/4+1+0.2) {$bc$};
	\node at (1/4+1,3/4,1/4) {$\circ$};
	\node at (1/4+1+0.3,3/4+0.1,1/4+0.2) {$cd$};
	
	\fill (3/8,-4/6+13/12, 17/24+1) circle (1.4pt) node[below] {$ab$};
	\fill (3/8+1, -4/6+13/12,17/24) circle (1.4pt) node[below] {$ad$};
	
	\node at (3/8+1/4, -4/6+13/12+3/4, 17/24+ 1/4) {$\circ$};
	\node at (3/8+1/4+0.25, -4/6+13/12+3/4+0.1, 17/24+ 1/4+0.2) {$0$};
	
	\end{scope}

	\begin{scope}[scale = 1.3,yshift = -1.5cm,scale=0.9, xshift = 3cm, rotate = -16]
	
	\draw (3/2,-1,5/2) -- (3,0,0);
	\draw (0,0,3) -- (0,3,0) -- (3/2,-1,5/2) -- (0,3,0) -- (3,0,0);
	\draw (0,0,3) -- (3/2,-1,5/2);
	
	\draw[dashed] (3,0,0) --  (3/2,3/3,3/2-3/3) -- (0,0,3);
	
	\draw[dashed] (3/2,3/3,3/2-3/3) -- (0,3,0);
	\draw[dashed] (3/2,-1,5/2) -- (3/2,3/3,3/2-3/3);
	

	\fill (3/2,-1,5/2) circle (2pt) node[below] {$a$};
	\fill (3,0,0) circle (2pt) node[right] {$d$};
	\fill (0,3,0) circle (2pt) node[left] {$p$};
	\fill (0,0,3) circle (2pt) node[left] {$b$};

	\fill (3/4,1,5/4) circle (1.4pt) node[right] {$a_1$};
	
	\node at (3/2,1,1/2) {$\circ$};
	\node at (3/2+ 0.25,1+0.1,3/2 - 1+0.2) {$c$};
	\node at (3/4,2,1/4) {$\circ$};
	\node at (3/4+0.3,2+0.1,1/4+0.2) {$c_1$};
	
	\node at (3/2,1/2,3/2-1/2) {$\circ$};
	\node at (3/2,0,3/2) {$\circ$};
	\node at(3/2, -1/2,3/2+1/2) {$\circ$};
	
	\node at (3/8,1,1/8+3/2) {$\circ$};
	\node at (3/8-0.2, 1+0.1,1/8+3/2+0.2) {$bc$};
	\node at (3/8+3/2,1,1/8) {$\circ$};
	\node at (3/8+3/2+0.3, 1+0.1,1/8+0.2) {$cd$};
	
	\fill (3/8,1/2,5/8+3/2) circle (1.4pt) node[below] {$ab$};
	\fill (3/8+3/2,1/2,5/8) circle (1.4pt) node[below] {$ad$};
	
	\node at (3/4, 3/2, 3/4) {$\circ$};
	\node at (3/4+0.25,3/2+0.1, 3/4+0.2) {$0$};
	
	\end{scope}

	\begin{scope}[scale = 1.3,yshift = 0.5cm, xshift = 6cm, scale =0.9, rotate = -16]
	
	\draw (3/2,-1,5/2) -- (3,1/2,-1/2);
	\draw (0,1/2,5/2) -- (0,3,0) -- (3/2,-1,5/2) -- (0,3,0) -- (3,1/2,-1/2);
	\draw (0,1/2,5/2) -- (3/2,-1,5/2);
	
	\draw[dashed] (3,1/2,-1/2) --  (3/2,3/3,3/2-3/3) -- (0,1/2,5/2);
	
	\draw[dashed] (3/2,3/3,3/2-3/3) -- (0,3,0);
	\draw[dashed] (3/2,-1,5/2) -- (3/2,3/3,3/2-3/3);
	

	\fill (3/2,-1,5/2) circle (2pt) node[below] {$a$};
	\fill (3,1/2,-1/2) circle (2pt) node[right] {$d$};
	\fill (0,3,0) circle (2pt) node[left] {$p$};
	\fill (0,1/2,5/2) circle (2pt) node[left] {$b$};

	\fill (3/4,1,5/4) circle (1.4pt) node[left] {$a_1$};
	
	\node at (3/2,1,1/2) {$\circ$};
	\node at (3/2+ 0.25,1+0.1,3/2 - 1+0.2) {$c$};
	\node at (3/4,2,1/4) {$\circ$};
	\node at (3/4+0.3,2+0.1,1/4+0.2) {$c_1$};
	
	\node at (3/2,1/2,3/2-1/2) {$\circ$};
	\node at (3/2,0,3/2) {$\circ$};
	\node at(3/2, -1/2,3/2+1/2) {$\circ$};
	
	\node at (3/8,1+1/4, 1/8+5/4) {$\circ$};
	\node at (3/8-0.2, 1+1/4+0.1,1/8+5/4+0.2) {$bc$};
	\node at (3/8+3/2,1+1/4,1/8-1/4) {$\circ$};
	\node at (3/8+3/2+0.3, 1+1/4+0.1,1/8-1/4+0.2) {$cd$};
	
	\fill (3/8,1/2+1/4,5/8+5/4) circle (1.4pt) node[below] {$ab$};
	\fill (3/8+3/2, 1/2+1/4,5/8-1/4) circle (1.4pt) node[below] {$ad$};
	
	\node at (3/4, 3/2, 3/4) {$\circ$};
	\node at (3/4+0.25,3/2+0.1, 3/4+0.2) {$0$};
	
	\end{scope}

	\begin{scope}[yshift = -6.5cm, scale=0.7]
	
	\fill (0,0) circle(2pt);
	\fill (3/2,2) circle (2pt);
	\fill (3,0) circle(2pt);
	
	\draw (0,0) -- (3,0) -- (3/2,2) -- (0,0);
	
	\node[left] at (0,0) {$n_1$};
	\node[left] at (3/2,2) {$n_2$};
	\node[right] at (3,0) {$n_3$};
	
	\fill (3/2+3/8,3/2) circle(2pt);
	\fill (3/2 + 6/8,1) circle(2pt);
	\fill (3/2 + 9/8,1/2) circle(2pt);
	
	\end{scope}
	
	\begin{scope}[yshift = -6.5cm, xshift=4.5cm, scale=0.7]
	
	\fill (3,0) circle(2pt);
	\fill (3,2) circle (2pt);
	\fill (-2,2) circle(2pt);
	\fill (-2,0) circle(2pt);
	
	\draw (3,0) -- (-2,0) -- (-2,2) -- (3,2) -- (3,0);
	
	\node[below] at (3,0) {$n_4$};
	\node[above] at (3,2) {$n_3$};
	\node[left] at (-2,2) {$n_2$};
	\node [below] at (-2,0) {$n_1$};
	
	\fill (+1/2,2) circle(2pt);
	\fill (+1/2,0) circle(2pt);
	
	\end{scope}
	
	\begin{scope}[yshift = -6.5cm, xshift=7.5cm, scale=0.7]
	
	\fill (0,0) circle(2pt);
	\fill (0,2) circle(2pt);
	\fill (6,2) circle(2pt);
	\fill (2,0) circle(2pt);
	
	\draw (0,0) -- (0,2) -- (6,2) -- (2,0) -- (0,0);
	
	\node[left] at (0,0) {$n_1$};
	\node[left] at (0,2) {$n_2$};
	\node[below] at (6,2) {$n_3$};
	\node [right] at (2,0) {$n_4$};
	
	\fill (2,2) circle(2pt);
	\fill (4,2) circle(2pt);
	
	\end{scope}

	\end{tikzpicture}
	\caption{The polytopes $c), \, d)$ and $e)$ and their cone $\sigma$ from left to right} \label{figure_3_polytopes_Todorov_type_surfaces}
\begin{align*}
c): \quad & \Delta = \langle a = (2,1,5), p = (-2,-1,-3), b=(2,0,1), d= (2,2,1) \rangle  \\ 
& F(\Delta) = \langle (0,0,0),(1,1/2,2),(1,1/4,1),(1,3/4,1) \rangle
\\
d) \quad & \Delta = \langle a=(2,-1,3), b=(2,0,1), c=(2,-1,-1), d= (2,-2,1), p = (-2,1,-1) \rangle \\
& F(\Delta) = \langle(0,0,0), (1,-1/2,1), (1,-1/2,0), (1,-3/4,1/2), (1,-1/4,1/2) \rangle \\
e) \quad & \Delta = \langle a = (2,0,1), b= (2,1,-1), c = (2,4,-3), d= (2,1,1) ,p = (-2,-2,1) \rangle
\\
& F(\Delta) = \langle(0,0,0), (1,3/2,-1),(1,3/4,0), (1,1/2,0),(1,3/4,-1/2) \rangle
\end{align*}
\end{figure}

Let $Y:= Z_{\Sigma}$ be a Todorov surface embedded in the toric $3$-fold $\mathbb{P}_{\Sigma}$ and consider the (rational) map $\psi_{nK_Y}$ associated to the complete linear system $\mathbb{P}H^0(Y, nK_Y)$. By (\cite[Prop.6.2]{Gie21}), since $Y$ sits in a toric $3$-fold, all these maps are in fact morphisms and in particular taking $n=2$ we get a morphism $\psi_{2K_Y}: Y \rightarrow \mathbb{P}^3$.

\begin{proposition}
	In all $3$ cases c), d) and e) the image $\psi_{2K_Y}(Y)$ is a quadric cone. The canonical divisor $K_Y$ defines a hyperelliptic curve of genus $3$ respectively. \\	
	In c) $Y$ has $2$ singularities of type $A_3$, in d) $Y$ has $4$ singularities of type $A_1$ and in e) $Y$ has $2$ singularities of type $A_2$. The fundamental group of $Y$ is $\mathbb{Z}/2 \mathbb{Z}$.
\end{proposition}

\begin{proof}
According to the pictures c), d) and e) in Construction \ref{specializing_facts_to_canonical_Fano_3_topes} if we let 
$ 2 \cdot F(\Delta)_{can} \cap M = \{ y_1,y_2,y_3 \}$ chosen as in the pictures, we get in all $3$ cases the relation $y_2^2 = y_1y_3$. Since by (\cite[Thm.4.2]{Gie21}) for $y_0 := (0,0,0)$
\[ H^0(Y,2K_Y) \cong \langle y_0, y_1, y_2,y_3 \rangle   \]
in all $3$ cases $\psi_{2K_Y}(Y) \subset \mathbb{P}^3$ is a quadric cone. The canonical curve $K_Y$ is hyperelliptic since in all cases we have a fibering $D_{can} \rightarrow \mathbb{P}^1$ given by projecting onto the axis $\langle a,c \rangle$ and the hypersurface $Y$ intersects a general fibre in $2$ points, as could be checked with a Newton polytope type argument. The computation of the singularities follows from Figure \ref{figure_3_polytopes_Todorov_type_surfaces} and the result on the fundamental group follows from (\cite{CD89}).
\end{proof}

\begin{table} 
	\caption{polytopes such that $\Delta_{can}$ has $3$ vertices sorted as in Figure \ref{figure_complete_list_polytopes_a} from the top to the bottom and from left to right. The arrows indicate that the polytopes are not canonically closed and the ID of the canonical closure is the polytope above the arrows (e.g. ID5389063 has canonical closure ID546219)} \label{table_spanning_set_polytope_Dcan_3vertices}
	\begin{tabular}{l}
		$ F(\Delta) = \langle (0,0,0),(1,1/3,0),(1,2/3,0),(1,1/2,-1/2) \rangle $ \\
		$p:=(-4,-2,1), a_2:=(-2,-1,0), c_2:=(-2,-1,1), b_1:=(-1,-1,1),$ \\
		$d_1:=(-1,0,1), a_1:=(0,0,-1), 0:=(0,0,0), c_1:=(0,0,1),$ \\
		$ab:=(1,0,0), bc:=(1,0,1), ad:=(1,1,0), cd:=(1,1,1)$, \\
		$b:=(2,0,1), a:=(2,1,-2), ac_1:=(2,1,-1), ac_2:=(2,1,0),$ \\
		$c:=(2,1,1), d:=(2,2,1)$
	\end{tabular}
	\begin{tabular}{|l|l|l|l|l|}
		\hline
		ID & spanning set for   & number of & sing. of   & sing. of    \\ 
		& polytope  $\Delta$       & points in $\Delta$ &  $Z_{\tilde{\Delta}}$ & $Z_{F(\Delta)}$  \\
		\hline
		547444 & $\Delta_{can}, p $ & 18  & $3 A_2$ & $3A_2$   \\ \hline
		474457 & $\Delta_{can}, a_2,c_2,d_1,b_1$ & 17 & $2A_2$ & $3A_2$        \\ \hline
		$\Rightarrow$ 545932 & $\Delta_{can}, a_2,c_2$ & 15 & &   \\ \hline
		$\Rightarrow$ 532384 & $\Delta_{can}, a_2,c_2,d_1$ & 16 & &   \\  \hline
		$\Rightarrow$ 532606 & $\Delta_{can}, a_2,d_1,b_1$ & 16 & &   \\  \hline
		483109 & $\Delta_{can},d_1,b_1,c_2,a_1$ & 16  & 
		$A_2, \, 3A_1 $ & $A_5,A_2$        \\ \hline
		534669 & $\Delta_{can}, c_2, d_1,a_1$ & 15 & $2A_2, \, A_1$ & $A_5,A_2$         \\ \hline
		534866 & $\Delta_{can}, b_1,a_1,d_1$ & 15 & $3A_2, \, A_1$ & $E_6,A_2$         \\ \hline
		534667 & $\Delta_{can}, c_2, d_1, b_1$ & 15  &$ 3A_2$ & $A_8$         \\ \hline
		546062 & $ \Delta_{can}, b_1,a_2 $ & 15 & $2 A_2, \, A_1$ & $3A_2$   \\ \hline
		546205 & $ \Delta_{can}, a_1, c_2 $ & 14 & $A_3, A_2 $ & $A_5,A_2$       \\ \hline
		546219 & $\Delta_{can}, c_1,a_2$ & 14 & $2A_2$ & $3A_2$         \\ \hline
		$\Rightarrow$ 547524 & $\Delta_{can}, a_2$ & 11 & & \\ \hline
		$\Rightarrow$ 546863 & $\Delta_{can}, a_2, bc$ & 12 & &  \\ \hline
		$\Rightarrow$ 539063 & $\Delta_{can}, a_2, bc,cd$ & 13 & &   \\ \hline
		536498 & $\Delta_{can}, b_1,ad,c_2$ & 14 & $A_3, A_2$ & $A_8$         \\ \hline
		537834 & $\Delta_{can}, ab,ad,c_2$ & 13 & $A_4$ & $A_8$         \\ \hline
		$\Rightarrow$ 547525 & $\Delta_{can}, c_2$ & 11 & &  \\ \hline
		$\Rightarrow$ 546862 & $\Delta_{can},ab,c_2$ & 12& & \\ \hline
		$\Rightarrow$ 546663 & $\Delta_{can}, ad,c_2$ & 12 & & \\ \hline

	\end{tabular}
\end{table}

\begin{table} 
	\caption{polytopes and Fine interior in the class $b)$ sorted as in Figure \ref{figure_complete_list_polytopes_b} from top to bottom, left to right. (with the same convention as in Table \ref{table_spanning_set_polytope_Dcan_3vertices})}. \label{table_spanning_set_polytope_Dcan_4vertices}
	\begin{tabular}{l}
		$F(\Delta) = \langle (0,0,0),(1,-1,1/2),(1,-2/3,1/3),(1,-1/2,1/2),(1,-2/3,2/3) \rangle $\\
		$p:=(-4,3,-2), c_2:=(-2,2,-1), a_2:=(-2,1,-1), b_1:=(-1,1,0)$ \\
		$d_1:=(-1,1,-1), 0:=(0,0,0), a_1:=(0,-1,0), c_1:=(0,1,0), cd:= (1,0,0),$ \\
		$ad:=(1,-1,0), ab:=(1,-1,1), bc:=(1,0,1), ac_2:= (2,-1,1), $ \\
		$ac_1:= (2,-2,1), d:= (2,-1,0), c:=(2,0,1), a:=(2,-3,1), b:= (2,-1,2)$ \\ \\
	\end{tabular}
	\begin{tabular}{|l|l|l|l|l|}
		\hline
		ID & spanning set for   & number of  & sing. of & sing. of  \\ 
		& polytope  $\Delta$       & points in $\Delta$ &  $Z_{\tilde{\Delta}}$ & $Z_{F(\Delta)}$  \\
		\hline
		545317 & $\Delta_{can},p$ & 18 & $3A_1$ & $3A_1$ 	  \\ \hline
		354912 & $\Delta_{can},c_2,a_2,d_1,b_1$ & 17 & $2A_1$ & $A_2, 2A_1$         \\ \hline
		$\Rightarrow$ 533513 & $\Delta_{can}, c_2,a_2$ & 15 & &  \\ \hline
		$\Rightarrow$ 481575 & $\Delta_{can}, c_2,a_2,d_1$ & 16 & &  \\ \hline
		372528 & $\Delta_{can},d_1,b_1,c_2,a_1$ & 16 & $3A_1$ & $A_4, A_1$          \\ \hline
		372973 & $\Delta_{can},b_1,d_1,a_2,c_1$ & 16 & $4A_1$ & $A_3, 2A_1$          \\ \hline
		$\Rightarrow$ 490511 & $\Delta_{can},b_1,d_1,a_2$ & 15 & &   \\ \hline
		388701 & $\Delta_{can},a_1,d_1,b_1,c_1$ & 15 & $4A_1$ & $D_5, A_1$          \\ \hline
		$\Rightarrow$ 499287 & $\Delta_{can}, a_1,d_1,b_1$ & 14 & &  \\ \hline
		490485 & $\Delta_{can},c_1,a_2,d_1$ & 15 & $3A_1$ & $A_3, 2A_1$         \\ \hline
		490481 & $\Delta_{can}, c_2,b_1,d_1$ & 15 & $ 2A_2$ & $A_6$          \\ \hline
		490478 & $\Delta_{can}, d_1,c_2,a_1$ & 15  & $3A_1$ & $A_4, A_1$         \\ \hline
		535952 & $\Delta_{can}, a_2, c_1$ & 14 & $3A_1$ & $A_3, 2A_1$        \\ \hline
		536013 & $\Delta_{can}, a_1,c_2$ & 14 & $A_2, \, A_1$ & $A_4, A_1$         \\ \hline
		495687 & $\Delta_{can},d_1,c_2,ab$ & 14 & $A_2, \, A_1$ & $A_6$         \\ \hline
		$\Rightarrow$ 539313 & $\Delta_{can},d_1,c_2$ & 13 & &   \\ \hline
		499291 & $\Delta_{can},c_1,b_1,d_1$ & 14 & $ A_3, 2A_1$ & $D_7$         \\ \hline
		$\Rightarrow$ 538356 & $\Delta_{can}, b_1, d_1$ & 13 & &   \\ \hline
		499470 & $\Delta_{can}, a_2,bc,d_1$ & 14 & $A_2, \, 2A_1$ & $A_4, 2A_1$       \\ \hline
		$\Rightarrow$ 539304 & $\Delta_{can},a_2,d_1$ & 13 & &  \\ \hline
		501298 & $\Delta_{can},c_2,ab,ad$ & 13 & $ A_2$ & $A_6$        \\ \hline
		$\Rightarrow$ 547246 & $\Delta_{can},c_2$ & 11 & & \\ \hline
		$\Rightarrow$ 540602 & $\Delta_{can}, c_2,ab$ & 12 &&  \\ \hline
		501330 & $\Delta_{can},a_2,bc,cd$ & 13  & $2A_1$ & $A_4, 2A_1$          \\ \hline
		$\Rightarrow$ 547240 & $\Delta_{can}, a_2$ & 11 & &   \\ \hline
		$\Rightarrow$ 540663 & $\Delta_{can}, a_2,bc$ & 12 & &   \\ \hline

	\end{tabular}
\end{table}


\newpage
	
	\begin{thebibliography}{XXX}
	
		\bibitem[Bat94]{Bat94} V. ~V. Batyrev, \emph{{Dual Polyhedra and Mirror Symmetry for Calabi-Yau Hypersurfaces
			in Toric Varieties}}, Journal of Algebraic Geometry \textbf{3} (1994), no.~3, 493--535.
	
		
			\bibitem[Bat17]{Bat17}
		V.\,V.\,Batyrev,
		\emph{ The stringy Euler number of Calabi-Yau hypersurface
			in toric varieties and the Mavlyutov duality}, Pure Appl. Math. Q. 
		{\bf 13} (2017), no. 1, 1–47.
		
		
		
		\bibitem[Bat20]{Bat20} V. \, V. \, Batyrev, \emph{{Canonical models of toric hypersurfaces}}, (2020), arXiv:2008.05814v1 [mathAG]
		

		\bibitem[Bomb73]{Bomb73} E. Bombieri, \emph{{Canonical models of surfaces of general type}}, Publications Math{\'e}matiques de l'IH{\'E}S, tome {\bf 42} (1973), 171--219.
		
	
		
		\bibitem[BHPV04]{BHPV04} W. Barth, K. Hulek, C. Peters, A. Van de Ven, \emph{{Compact Complex Surfaces}}, Ergebnisse der Mathematik und ihrer Grenzgebiete. 3. Folge / A Series of Modern Surveys in Mathematics, (2004).
		
		

		
		\bibitem[BKS19]{BKS19} V.\,V.\, Batyrev, A.\,M.\,Kasprzyk,
		and K. Schaller,
		\emph{On the Fine interior of three-dimensional canonical 
		Fano polytopes}, arXiv:1911.12048 [math.AG]
	
	
	
		\bibitem[Cat78]{Cat78} F. Catanese, \emph{{Surfaces with $K^{2}=p_{g}=1$ and their period mapping}}, Algebraic Geometry Summer Meeting, Copenhagen, (1978), 1 -- 29.
	
	
		\bibitem[CD89]{CD89}
		F. Catanese, O. Debarre, \emph{{Surfaces with $K^2 = 2, p_g = 1, q= 0$.}}, (1989).

	
		
		\bibitem[CLS11]{CLS11}
		D.~A. Cox, J.~B. Little and H. K. Schenck,
		\emph{Toric varieties}, Graduate Studies in Mathematics, 124,
		Amer. Math. Soc., Providence, RI, (2011).
		

		
	
		\bibitem[Gie21]{Gie21} J. Giesler, \emph{{The plurigenera and birational models of nondegenerate toric hypersurfaces}}.
	
	
		
		\bibitem[Kan76]{Kan76}
		V. Kanev, \emph{{An example of a simply connected surface of general type for which the local torelli theorem does not hold }} (1976).
		
		
		\bibitem[Kas10]{Kas10}
		A. M. Kasprzyk, \emph{Canonical toric Fano threefolds},
		Can. J. Math. 62 (6) (2010) 1293--1309.
		
		

		
		\bibitem[KM98]{KM98} 
		J. Koll{\'a}r and S. Mori, \emph{Birational Geometry of Algebraic Varieties}, Vol. 134 of Cambridge Tracts in Mathematics, Cambridge University Press, Cambridge (1998).
		
		


	
		
		
		\bibitem[Rei83]{Rei83}  M. Reid, \emph{Decomposition of toric morphisms},
		in Arithmetic and geometry,  Progr. Math. {\bf 36}, Birkh\"auser Boston,
		Boston, MA (1983), 395--418.
		
		
		\bibitem[Rei87]{Rei87}  M. Reid, 
		\emph{Young person’s guide to canonical singularities}, Algebraic geometry, Bowdoin, 1985
		(Brunswick, Maine, 1985), Proc. Sympos. Pure Math., vol. 46, Amer. Math. Soc., Providence, RI,
		(1987), pp. 345--414
		
		


		\bibitem[Sch18]{Sch18} K. Schaller,
		\emph{{Stringy Invariants of Algebraic Varieties and Lattice Polytopes}}, Ph.D. thesis, Eberhart-Karls-Universit\"at T\"ubingen, (2018).



		\bibitem[Tod80]{Tod80} A. Todorov,
		\emph{{Surfaces of general type with $p_g=1$ and $(K.K) = 1$.I}}, Annales scientifiques de l'{\'E}cole Normale Sup{\'e}rieure, Serie 4, Volume 13 no. 1, (1980), 1-21.

		
	\bibitem[Tre10]{Tre10} J. Treutlein,
	\emph{{Birationale Eigenschaften generischer Hyperflächen in algebraischen Tori}}, Dissertation (2010).



		\bibitem[Us87]{Us87} S. Usui,
		\emph{{Type I Degeneration of Kunev surfaces}}, Th{\'e}orie de Hodge - Luminy, Juin 1987, Ast{\'e}risque no. {\bf 179-180}, 183--243.
		

		
	
		

	\end{thebibliography}
\end{document}